%% file: main-Merged.tex
\documentclass[11pt, letter]{article}
\usepackage[english]{babel}
\usepackage[utf8x]{inputenc}
\usepackage[T1]{fontenc}

\usepackage{fullpage}

\usepackage{mathrsfs}
\usepackage{amsmath}
\usepackage{amssymb}
\usepackage{graphicx}
\usepackage[colorinlistoftodos]{todonotes}
\usepackage{hyperref}
\hypersetup{colorlinks=true, allcolors=blue}
\usepackage{amsthm}
\usepackage{slashbox}

\newtheorem{theorem}{Theorem}[section]
\newtheorem{lemma}[theorem]{Lemma}
\newtheorem{corollary}[theorem]{Corollary}

\newtheorem{lem}[theorem]{Lemma}

\theoremstyle{remark}

\usepackage{fancyhdr,lastpage}
\usepackage{amsfonts}

\usepackage{color}

\usepackage[ruled,vlined]{algorithm2e}

\newcommand{\scr}[0]{\mathscr}
\newcommand{\bb}[0]{\mathbb}

\newcommand{\mc}[1]{\mathcal{#1}}

\newcommand{\fn}[2]{#1\left(#2\right)}

\newcommand{\set}[1]{\left\{#1\right\}}
\newcommand{\floor}[1]{\left\lfloor#1\right\rfloor}

\newcommand{\clInt}[2]{\left[#1,#2\right]}
\newcommand{\smallo}[1]{o\left(#1\right)}

\usepackage[numbers]{natbib}

\usepackage{enumerate}

\author{Ndiam\'e Ndiaye\thanks{Department of Mathematics \& Statistics, McGill University: ndiame.ndiaye@mail.mcgill.ca}\\
\and Sergey Norin\thanks{Department of Mathematics \& Statistics, McGill University: sergey.norin@mcgill.ca}\\
\and Adrian Vetta\thanks{School of Computer Science and Dept. of Mathematics \& Statistics, McGill University: adrian.vetta@mcgill.ca}}

\title{Descending the Stable Matching Lattice: How many\\ Strategic Agents are required to turn Pessimality to Optimality?}

\date{}

\begin{document}

\pagestyle{plain}

\maketitle

\vspace{-.7cm}

\begin{abstract}
\input{abstract-Merged.tex}
\end{abstract}

\input{introduction-Merged.tex}

\input{background-Merged.tex}

\input{incentives-Merged.tex}

\newpage
\input{Example-Merged.tex}

\newpage
\input{Models-Merged.tex}

\input{Minimum-Merged.tex}

\input{Random-Merged.tex}

\input{conclusion-Merged.tex}



\bibliographystyle{plain}
\bibliography{references}


\end{document}

%% file: abstract-Merged.tex
The set of stable matchings induces a distributive lattice. The supremum of the stable matching lattice is the 
boy-optimal (girl-pessimal) stable matching and the infimum is the girl-optimal (boy-pessimal) stable matching.
The classical boy-proposal deferred-acceptance algorithm returns the supremum of the lattice, that is, the boy-optimal stable matching.
In this paper, we study the smallest group of girls, called the {\em minimum winning coalition of girls},
that can act strategically, but independently, to force the boy-proposal deferred-acceptance algorithm to
output the girl-optimal stable matching.
We characterize the minimum winning coalition in terms of stable matching rotations and show that its cardinality 
can take on any value between $0$ and $\floor{\frac{n}{2}}$, for instances with $n$ boys and $n$ girls. 
Our two main results concern the random matching model. First, the expected cardinality of the minimum winning coalition
is small, specifically $(\frac{1}{2}+o(1))\log{n}$. This resolves a conjecture of Kupfer~\cite{Kup18}. 
Second, in contrast, a randomly selected coalition must contain nearly every girl to ensure it is a winning coalition asymptotically almost surely.
Equivalently, for any $\varepsilon>0$, the probability a random group of $(1-\varepsilon)n$ girls is {\em not} a 
winning coalition is at least $\delta(\varepsilon)>0$.

%% file: introduction-Merged.tex
\section{Introduction}
We study the stable matching problem with $n$ boys and $n$ girls. Each boy has a preference ranking over the girls and
vice versa. A matching is {\em stable} if there is no boy-girl pair that prefer each other over their current partners in the matching.
A stable matching always exists and can be found by the deferred-acceptance algorithm \cite{GS62}.
Furthermore, the set of stable matchings forms a lattice whose supremum matches each boy to his {\em best} stable-partner
and each girl to her {\em worst} stable-partner. This matching is called the {\em boy-optimal} (girl-pessimal) stable matching.
Conversely, the infimum of the lattice matches each boy to his worst stable-partner and each girl to her best stable-partner. 
Consequently this matching is called the {\em girl-optimal} (boy-pessimal) stable matching. 

Interestingly, the deferred-acceptance algorithm outputs the optimal stable matching for the proposing side. Perhaps surprisingly, 
the choice of which side makes the proposal can make a significant difference.
For example, for the random matching model, where the preference list of each boy and girl is sampled uniformly and independently,
Pittel~\cite{Pit89} showed the boy-proposal deferred acceptance algorithm assigns the boys with much better ranking partners than the girls.
Specifically, with high probability, the sum of the partner ranks is close to $n\log n$ for the boys and close to $\frac{n^2}{\log n}$ for the girls. 
Hence, on average, each boy ranks his partner at position $\log n$ at the boy-optimal stable matching while
each girl only ranks her partner at position $\frac{n}{\log n}$.
Consequently, collectively the girls may have a much higher preference for the infimum (girl-optimal) stable matching than the supremum
(girl-pessimal) stable matching output by the boy-proposal deferred-acceptance algorithm.

Remarkably, Ashlagi et al.~\cite{AKL17} proved that in an {\em unbalanced market} with one fewer girls than boys
this advantage to the boys is reversed. In the random matching model, with high probability, each girl is matched to a 
boy she ranks at $\log n$ on average and each boy is matched to a girl he ranks at $\frac{n}{\log n}$ on average, even using the
boy-proposal deferred-acceptance algorithm.\footnote{In fact,
an unbalanced market essentially contains a unique stable matching; see~\cite{AKL17} for details.}
Kupfer~\cite{Kup18} then showed a similar effect arises in a balanced market in which exactly one
girl acts strategically. The expected rank of the partner of each girl improves to $O(\log^4 n)$ while the
expected rank of the partner of each boy deteriorates to $\Omega(\frac{n}{\log^{2+\epsilon} n})$. 
Thus, just one strategic girl suffices for the stable matching output by the boy-proposal deferred-acceptance algorithm
to change from the supremum of the lattice to a stable matching ``close'' to the infimum. But how many strategic girls are required to
guarantee the infimum itself is output? Kupfer~\cite{Kup18} conjectured that $O(\log n)$ girls suffice in expectation.
In this paper we prove this conjecture. More precisely, we show that the minimum number of strategic girls required 
is $\frac12\log n +O(\log\log n) = (\frac12 +o(1))\log n$ in expectation. 
Consequently, the expected cardinality of the optimal winning coalition of girls is relatively small.

Conversely, a random coalition of girls must be extremely large, namely of cardinality $n-o(n)$, if it is to be
a winning coalition with high probability. We prove that, for any $\varepsilon>0$, the probability a random group of $(1-\varepsilon)n$ girls is {\em not} a 
winning coalition is at least a constant.

\subsection{Overview}

In Section~\ref{sec:background}, we present the relevant background on the stable matching problem, in particular, concerning 
the stable matching lattice and the rotation poset. In Section~\ref{sec:incentives} we provide a characterization of
winning coalitions of girls in terms of minimal rotations in the rotation poset.
This allows us to show that for general stable matching instances the cardinality of the minimum winning coalition may 
take on every integral value between a lower bound of $0$ and an upper bound of $\floor{\frac{n}{2}}$.
In Section~\ref{sec:ex}, we present an example to 
illustrate the relevant stable matching concepts and ideas used in the paper. 
In Section~\ref{sec:random-model}, we present the random matching model studied for the main 
results of the paper. Our first main result is given in Section~\ref{sec:minimum-coalitions} 
and shows that in random instances 
the cardinality of the minimum winning coalition is much closer to the lower bound than the upper bound.
Specifically, in the random matching model, the expected cardinality of the minimum winning coalition
is  $\frac12\log n +O(\log\log n)$. 
Our second main result is presented in Section~\ref{sec:random-coalitions} 
and shows that for a randomly selected coalition to be a winning coalition with probability $1-o(1)$, 
it must have cardinality $n-o(n)$.

%% file: background-Merged.tex
\section{The Stable Matching Problem}\label{sec:background}

Here we review the stable matching problem and the concepts and results relevant to this paper.
The reader is referred to the book~\cite{GI89} by Gusfield and Irving for a comprehensive 
introduction to stable matchings. 
An illustrative example, defined in Table~\ref{tab:preferences}, will also be presented in Section~\ref{sec:ex}.
We remark that this example is deferred until all the relevant concepts have been 
defined (indeed, it will be clearer to present the structural properties of the example in a different order 
than how they are defined in this review).

We are given a set $B=\{b_1,b_2,\dots, b_n\}$ of boys and a set $G=\{g_1,g_2,\dots, g_n\}$ of girls. 
Every boy $b\in B$ has a preference ranking $\succ_b$ over the girls; similarly, every girl $g\in G$ 
has a preference ranking~$\succ_g$ over the boys. 
Now let $\mu$ be a (perfect) matching between the boys and girls. 
We say that boy $b$ is matched to girl $\mu(b)$ in the matching $\mu$; similarly, girl $g$ is matched to boy $\mu(g)$.
Boy~$b$ and girl~$g$ form a {\em blocking pair} $\{b,g\}$ if they prefer each other to their partners
in the matching~$\mu$; that is $g\succ_b \mu(b)$ {\bf and} $b\succ_g \mu(g)$.\
A matching $\mu$ that contains no blocking pair is called {\em stable}; otherwise it is unstable.
In the {\em stable matching problem}, the task is to find a stable matching.

\subsection{The Deferred-Acceptance Algorithm}\label{sec:DFA}
The first question to answer is whether or not a stable matching is guaranteed to exist.
Indeed a stable matching always exists, as shown in the seminal work of Gale and Shapley~\citep{GS62}. 
Their proof was constructive; the {\em deferred-acceptance algorithm}, described in Algorithm~\ref{alg:DFA},
outputs a stable matching.

\

\begin{algorithm}
\While{there is an unmatched boy $b$}{
Let $b$ propose to his favourite girl $g$ who has not yet rejected him\;
  \uIf{$g$ is unmatched}{
    $g$ provisionally matches with $b$\;
  }
  \uElseIf{$g$ is provisionally matched to $\hat{b}$}{
    $g$ provisionally matches to her favourite of $b$ and $\hat{b}$, and rejects the other\;
  }
  }
\caption{Deferred-Acceptance (Boy-Proposal Version)}\label{alg:DFA}
\end{algorithm}

\

The key observation here is that only a girl can reject a provisional match. Thus, from a girl's perspective, her provisional match 
can only improve as the algorithm runs. It follows that the deferred-acceptance algorithm terminates when every girl has received 
at least one proposal. In addition, from a boy's perspective, his provisional match can only get worse as the algorithm runs.
Indeed, it would be pointless for a boy to propose to girl who has already rejected him.
Thus, each boy will make at most $n$ proposals. Furthermore, because each boy makes proposals in decreasing order of preference,
every girl must eventually receive a proposal. Thus the deferred-acceptance algorithm must terminate with a perfect matching $\mu$.
At this point all provisional matches are made permanent.
But why will this permanent set of matches $\mu$ form a stable matching? 
The proof is simple and informative, so we include it for completeness.

\begin{theorem}[Gale and Shapley 1962 ~\citep{GS62}]
The deferred-acceptance algorithm outputs a stable matching.
\end{theorem}
\begin{proof}
Suppose $\{b, g\}$ is a blocking pair for $\mu$. Then boy $b$ prefers girl $g$ over girl $\hat{g}=\mu(b)$, that is $g \succ_{b} \mu(b)$. 
So $b$ must have proposed
to $g$ before proposing to $\hat{g}$. Then $g$ must have rejected $b$. Either she rejected $b$ at the time of the proposal or 
she provisionally accepted his offer but later
rejected him after receiving a better offer. As her provisional partner only improves over time, it follows that girl $g$ prefers 
her final permanent partner $\hat{b}=\mu(g)$ over $b$.
That is, $\mu(g) \succ_{g} b$, and so $\{b, g\}$ is not a blocking pair.
\end{proof}

\subsection{The Stable Matching Lattice}\label{sec:SM-lattice}
So a stable matching always exists.
In fact, there may be an exponential number of stable matchings~\cite{Knu82}; see Theorem~\ref{thm:all-integers} for an example.
The set $\mathcal{M}$ of all stable matchings forms a poset $(\mathcal{M}, \geqslant)$ whose order $\geqslant$ is defined via the
preference lists of the boys. Specifically, $\mu_1 \geqslant \mu_2$ if and only if every boy weakly prefers their partner
in the stable matching $\mu_1$ to their partner in the stable matching $\mu_2$; that is $\mu_1(b) \succeq_b \mu_2(b)$, for every boy~$b$.

Conway (see Knuth~\cite{Knu82}) observed that the poset $(\mathcal{M}, \geqslant)$ is in fact a {\em distributive lattice}.
Thus,  by the lattice property, each pair of stable matchings $\mu_1$ and $\mu_2$ has a {\em join} (least upper bound) and a {\em meet} 
(greatest lower bound) in the lattice. Moreover, the join $\hat{\mu} = \mu_1\vee \mu_2$ 
has the remarkable property that each boy $b$ is matched to his {\em most preferred} partner amongst
the girls $\mu_1(b)$ and $\mu_2(b)$. 
Similarly, in the meet $\check{\mu}= \mu_1\wedge \mu_2$ each boy is matched to his {\em least preferred} partner 
amongst the girls $\mu_1(b)$ and $\mu_2(b)$. 
In particular, in the {\em supremum} ${\bf 1}=\bigvee_{\mu\in \mathcal{M}} \mu$ of the lattice
each boy is matched to his most preferred partner from any stable matching (called his {\em best stable-partner}). 
Accordingly, the matching ${\bf 1}$ is called the {\em boy-optimal} stable matching.
On the other hand, in the {\em infimum} ${\bf 0}=\bigwedge_{\mu\in \mathcal{M}} \mu$ of the lattice
each boy is matched to his least preferred partner from any stable matching (called his {\em worst stable-partner}).
Accordingly, the matching ${\bf 0}$ is called the {\em boy-pessimal} stable matching.

\begin{theorem}~\citep{GS62}\label{thm:boy-optimal} 
The deferred-acceptance algorithm outputs the boy-optimal stable matching.
\end{theorem}
\begin{proof}
If not, let $b$ be the first boy rejected by a stable partner, say $g$, during the course of the deferred-acceptance algorithm.
Assume there is a stable matching $\hat{\mu}$ in which the pair $(b, g)$ is matched and assume that $g$ rejects $b$ in favour 
of the boy $\hat{b}$. By assumption $b$ was the first boy rejected by a stable partner so boy $\hat{b}$ had not been rejected by any stable-partner when $g$ rejected $b$.
Thus $\hat{b}$ prefers $g$ over any stable partner. In particular, he prefers $g$ over his stable partner
$\hat{\mu}(b)\neq g$. But then  $g \succ_{\hat{b}} \hat{\mu}(\hat{b})$ and $\hat{b}\succ_{g} \hat{\mu}(g)=b$.
Hence, $\{\hat{b} ,g\}$ is a blocking pair for the matching $\hat{\mu}$, a contradiction.
\end{proof}
The reader may have observed that the description of the deferred-acceptance algorithm given in Algorithm~\ref{alg:DFA} is ill-specified.
In particular, which unmatched boy is selected to make the next proposal? 
Theorem~\ref{thm:boy-optimal} explains the laxity of our description. It is irrelevant which unmatched boy is chosen in each step,
the final outcome is guaranteed to be the boy-optimal stable matching! In fact, the original description of the algorithm by 
Gale and Shapley~\citep{GS62}
allowed for simultaneous proposals by unmatched boys -- again this has no effect on the stable matching output.

The inverse poset $(\mathcal{M}, \leqslant)$ is also of fundamental interest. Indeed, McVitie and Wilson~\cite{MW71}
made the surprising observation that $(\mathcal{M}, \leqslant)$ is the lattice defined using the
preference lists of the girls rather than the boys. That is, every boy weakly prefers their partner
in the stable matching $\mu_1$ to their partner in the stable matching $\mu_2$ if and only if 
every girl weakly prefers their partner in the stable matching $\mu_2$ to their partner in the stable matching $\mu_1$.

\begin{theorem}~\citep{MW71}\label{thm:inverse} 
If $\mu_1\geqslant \mu_2$ in the lattice $(\mathcal{M}, \geqslant)$ then every girl weakly prefers $\mu_2$ over $\mu_1$.
\end{theorem}
\begin{proof}
Assume there is a girl $g$ who prefers boy $b=\mu_1(g)$ over boy $\mu_2(g)$. But, by assumption,
boy $b$ prefers $g=\mu_1(b)$ over girl $\mu_2(b)$. Thus $\{b,g\}$ is a 
blocking pair for the matching $\mu_2$, a contradiction.
\end{proof}
Consequently, the boy-optimal stable matching ${\bf 1}$ is also the {\em girl-pessimal} stable matching
and the boy-pessimal stable matching ${\bf 0}$ is the {\em girl-optimal} stable matching. 

For our example, the set of stable matchings and the stable matching lattice are shown in Table~\ref{tab:23-stable-matchings}
and Figure~\ref{fig:sm-lattice} of Section~\ref{sec:ex}, respectively.

 \subsection{The Rotation Poset}\label{sec:rotation-poset}
 
 Recall that the lattice $(\mathcal{M}, \geqslant)$ is a {\em distributive} lattice. 
 This is important because the {\em fundamental theorem for finite distributive
 lattices} of Birkhoff~\cite{Birk37} states that associated with any distributive lattice $\mathcal{L}$ is a unique
 {\em auxiliary poset} $\mathcal{P}(\mathcal{L})$. Specifically, the order ideals (or down-sets) of the auxiliary poset $\mathcal{P}$,
 ordered by inclusion, form the lattice $\mathcal{L}$. We refer the reader to the book of Stanley~\cite{Stan97} for details on the 
 fundamental theorem for finite distributive lattices. For our purposes, however, it is sufficient to note that the
 auxiliary poset $\mathcal{P}$ for the stable matching lattice $(\mathcal{M}, \geqslant)$ has an elegant
 combinatorial description that is very amenable in studying stable matchings.
 
 In particular, the auxiliary poset for the stable matching lattice is called the {\em rotation poset} $\mathcal{P}=(\mathcal{R}, \geq)$ 
 and was first discovered by Irving and Leather~\cite{IL86}. The elements of the auxiliary poset are {\em rotations}.
 Informally, given a stable matching $\mu$, a rotation will rearrange the partners of a suitably chosen subset of the boys in a circular fashion
 to produce another stable matching. 
Formally, a rotation $R\in \mathcal{R}$ is a subset of the pairs in the stable matching $\mu$,
$R=[(b_0,g_0), (b_1,g_1),\dots, (b_k, g_k)]$, such that
for each boy $b_i$, the girl $g_{i+1\, (\text{mod}\, k+1)}$ is the first girl {\em after} 
his current stable-partner $g_i$ on his preference list who would accept a proposal from him.
That is, $g_{i+1}$ prefers boy $b_i$ over her current partner boy $b_{i+1}$ {\em and} every girl $g$
that boy $b_i$ ranks on his list between $g_{i}$ and $g_{i+1}$ prefers her current partner in $\mu$ over $b_i$.

In this case, we say that $R$ is a {\em rotation exposed} by the stable matching $\mu$. 
Let $\hat{\mu}=\mu \otimes R$ be the perfect matching obtained by matching boy $b_i$ with 
the girl $g_{i+1\, (\text{mod}\, k+1)}$, for each $0\le i\le k$, with all other matches the same as in $\mu$.
Irving and Leather~\cite{IL86} showed that $\hat{\mu}$ is also a stable matching. More importantly they proved:
\begin{theorem}~\cite{IL86}\label{thm:covered} 
The matching $\hat{\mu}$ is covered\footnote{We say $y$ is {\em covered} by $x$ in a poset if $x \geqslant y$ and
there is no element $z$ such that $x \geqslant z \geqslant y$.} 
by $\mu$ in the Hasse diagram of the stable matching lattice
if and only if $\hat{\mu}=\mu \otimes R$ for some rotation $R$ exposed by $\mu$. \qed
\end{theorem}
Theorem~\ref{thm:covered} implies that we may traverse the stable
matching lattice $(\mathcal{M}, \geqslant)$ using rotations.
As stated, we may also derive a poset $\mathcal{P}=(\mathcal{R}, \geq)$ whose elements are rotations.
Let $\mathcal{R}_\mu$ be the set of all rotations exposed in $\mu$. Then $\mathcal{R} = \bigcup_{\mu\in\mathcal{M}}\mathcal{R}_\mu$ 
is the set of all rotations. We then define the partial order $\geq$ as follows.
Let $R_1\geq R_2$ in $\mathcal{P}$ {\em if and only if} for any stable matching $\mu_1\in\set{\mu\in\mathcal{M}: R_1\in \mathcal{R}_{\mu}}$ 
and any stable matching 
$\mu_2\in \set{\mu\in\mathcal{M}: R_2\in \mathcal{R}_\mu}$, either $\mu_1$ and $\mu_2$ are incomparable 
or $\mu_1\geqslant \mu_2$ in $(\mathcal{M}, \geqslant)$. 
This rotation poset $\mathcal{P}=(\mathcal{R}, \geq)$ is the auxiliary poset for the stable matching 
lattice $(\mathcal{M}, \geqslant)$; see Gusfield and Irving~\cite{GI89}. In particular, there is a bijection between stable matchings 
and {\em antichains} of the rotation poset.

The set of rotations and the rotation poset for our running example are illustrated in Table~\ref{tab:10-rotations} and Figure~\ref{fig:rotation-poset}
of Section~\ref{sec:ex}, respectively.

We remark that, unlike the stable matching lattice, the cardinality of the rotation poset is always polynomial.
Specifically, any boy-girl pair $\{b,g\}$ can appear in at most one rotation~\cite{IL86}. It immediately follows that
the rotation poset has at most $O(n^2)$ elements; in fact, Gusfield~\cite{Gus87} showed how to find all the
rotations in $O(n^2)$ time.

\subsection{The Rotation Graph}\label{sec:rotation-graph}

For any stable matching $\mu=\{ (b_1,g_1), (b_2,g_2),\dots,(b_n, g_n)\}$
we define an auxiliary directed graph $H(\mu)$. This graph, which we call the {\em (exposed) rotation graph}, has a vertex~$i$ for each boy $b_i$.
There is an arc from $i$ to $j$ if the next girl on $b_i$'s list to prefer $b_i$ over her current partner is $g_j$. 
If for some~$b_i$, no such girl exists, then $i$ has out-degree 0; otherwise it has  out-degree 1.
By definition, the rotations exposed in $\mu$ are exactly the cycles of $H(\mu)$.
(See Figure~\ref{fig:rotation-graph} in Section~\ref{sec:ex} for the rotation-graph $H({\bf 1})$ for the running example.)
For example, if $\mu={\bf 1}$ then $H({\bf 1})$ consists of the set of rotations 
exposed in the boy-optimal stable matching. We call these the {\em maximal rotations}. 

A rotation $R$ exposed in $\mu$ is {\em minimal} if $\mu\otimes R={\bf 0}$. Equivalently, the {\em minimal rotations}
are the set of rotations exposed in the girl-optimal stable matching $\bf 0$ when ordering using the preferences of the girls 
rather than the boys.

%% file: incentives-Merged.tex
\section{Incentives in the Stable Matchings Problem}\label{sec:incentives}

Intuitively, because the deferred-acceptance algorithm outputs the boy-optimal stable matching,
there is no incentive for a boy not to propose to the girls in order of preference.
This fact was formally proven by Dubins and Freedman~\cite{DF81}.
On the other hand, because the stable matching is girl-pessimal, it can be beneficial for
a girl to strategize. Indeed, Roth~\cite{Roth82} showed that no stable matching mechanism exists that is 
incentive compatible for every participant.

\subsection{The Minimum Winning Coalition of Girls}
\label{sec:Minimum coalition}
The structure of the stable matching lattice $\mathcal{L}$ is extremely useful in understanding the
incentives that arise in the stable matching problem. For example, the following structure will be of importance in this
paper. Let $F\subseteq G$ be a group of girls and let $\mathcal{M}_F$ be the collection of stable matchings where every girl in~$F$ 
is matched to their best stable-partner. Given the aforementioned properties of the join and meet operation
in the stable matching lattice, it is easy to verify that $\mathcal{L}_F=(\mathcal{M}_F, \geqslant)$ is also a lattice.
Thus, $\mathcal{L}_F$ has a supremum ${\bf 1}_F$ which is the boy-optimal stable matching given that every girl in $F$ is matched to their
best stable-partner. Similarly, $\mathcal{L}_F$ has a infimum ${\bf 0}_F$ which is the boy-pessimal stable matching
given that every girl in $F$ is matched to their best stable-partner.  Observe that ${\bf 0}_F$ is the girl-optimal stable-matching ${\bf 0}$, for any
subset $F$ of the girls.

Why is this useful here? Well, imagine that each girl in $F$ rejects anyone who is not their best stable-partner.
Then the deferred-acceptance algorithm will output the stable matching ${\bf 1}_F$; see also the works of Gale and Sotomayor~\cite{Sot85} 
on strong equilibria and of Gonczarowski~\cite{Gon14} on blacklisting. Of course, if $F=G$ 
then both ${\bf 1}_G$ and ${\bf 0}_G$ must match every girl to their optimal stable partner so ${\bf 1}_G={\bf 0}_G={\bf 0}$.

We will call any $F\subseteq G$ such that ${\bf 1}_F={\bf 0}$ a {\em winning coalition} and the smallest such group is 
called a {\em minimum winning coalition}. Winning coalitions can be found using the rotation poset.

\begin{theorem}\label{thm:winning}
A set of girls is a winning coalition if and only if it contains at least one girl from each
minimal rotation in the rotation poset $(\mathcal{R}, \geq)$
\end{theorem}
\begin{proof}
Let $G^2$ be the set of girls who have at least two stable-partners.
For each girl $g_j\in G^2$, let $\mathcal{M}_j$ be the set of stable matchings in which she is {\bf not} matched to her
best stable-partner. Then $\mathcal{L}_j=(\mathcal{M}_j, \geqslant)$ is a lattice with supremum ${\bf 1}^j$ and infimum ${\bf 0}^j$.
Observe that ${\bf 1}^j={\bf 1}$ and ${\bf 0}^j\neq {\bf 0}$.

Now let $\{\mu_1,\mu_2,\dots, \mu_k\}$ be the minimal stable-matchings in the poset $(\mathcal{M}\setminus {\bf 0}, \geqslant)$.
That is, $\{\mu_1,\mu_2,\dots, \mu_k\}$ is the set of matchings such that for any $i\in \set{1,...,k}$ and any stable matching $\mu\notin \{\mu_1,\mu_2,\dots, \mu_k\}\cup \{{\bf 0}\}$, there is a boy who strictly prefers $\mu$ over $\mu_i$.
For each $1\le \ell\le k$, observe that $\mu_\ell={\bf 0}^j$ for some girl $g_j\in G^2$. 
But, if $\mu_\ell={\bf 0}^j$ then girl $g_j$ must be matched to her best stable-partner in $\mu_i$ for any $i\neq \ell$.
Otherwise, because $\mu_\ell \wedge \mu_i = {\bf 0}$ in the stable matching lattice $\mathcal{L}$, it would be the case
that girl $g_j$ is not matched to her best stable-partner in ${\bf 0}$, a contradiction.
Let $U_{\ell}$ be the set of girls who are not matched to their best stable-partner in $\mu_\ell$.
Thus, the sets $U_1$,...,$U_k$ are non-empty and disjoint. 

Let $F$ be any group of girls that contains at least one girl from each set $U_{\ell}$, for $1\le \ell\le k$.
We claim that ${\bf 1}_F={\bf 0}$ and, consequently, $F$ is a winning coalition. For each $1\le \ell\le k$, at least one girl $g_j$ in $F$ is
not matched to best stable-partner in $\mu_\ell$. Thus $\mu_\ell = {\bf 0}^j > {\bf 0}$. It follows immediately that there is
a unique stable-matching, namely the girl-optimal stable matching ${\bf 0}$, that matches every girl in $F$ to their
best stable-partner. Hence, ${\bf 1}_F={\bf 0}$ as claimed.

Conversely, let $F$ be any group of girls such that for some $\ell$, $F\cap U_{\ell}=\emptyset$. We claim that ${\bf 1}_F\neq {\bf 0}$ and, 
consequently, $F$ is a losing coalition. By definition, every girl in $F$ is matched to their best stable-partner in $\mu_\ell$. But then, by definition, ${\bf 1}_F\geqslant \mu_\ell > {\bf 0}$ as claimed.

Finally, observe that our definition of $U_{i}$ is exactly the set of girls in the unique rotation exposed in $\mu_i$ 
which is a minimal rotation of the rotation poset $(\mathcal{R}, \geq)$ which proves the statement of the theorem.
\end{proof}
Theorem~\ref{thm:winning} allows us to find a minimum winning coalition.
\begin{corollary}\label{cor:minimum-size}
The cardinality of the minimum winning coalition is equal to the cardinality of the set of 
minimal rotations in the rotation poset $(\mathcal{R}, \geq)$.
\end{corollary}

Section~\ref{sec:ex} provides an illustration of how rotations correspond to stable matchings and gives a 
minimum winning coalition for the running example.

\subsection{Efficiency and Extremal Properties}

From the structure inherent in Theorem~\ref{thm:winning} and Corollary~\ref{cor:minimum-size} 
we can make several straight-forward deductions regarding winning coalitions.

First, Theorem~\ref{thm:winning} implies that we have a polynomial algorithm to verify winning coalitions.
Likewise Corollary~\ref{cor:minimum-size} implies that we have a polynomial time algorithm to compute
the minimum winning coalition. In fact, the techniques of Gusfield~\cite{Gus87}
(see also~\cite{GI89}) can now be used to solve both problems in $O(n^2)$ 
time.

Second, we can upper bound the cardinality of the minimum winning coalition.

\begin{lemma}\label{lem:minimum-size}
In any stable matching problem the  minimum winning coalition has cardinality at most $\floor{\frac{n}{2}}$.
\end{lemma}
\begin{proof}
Consider the minimal stable-matchings $\{\mu_1,\mu_2,\dots, \mu_k\}$ in the poset $(\mathcal{M}\setminus {\bf 0}, \geqslant)$.
We claim $k\le \lfloor \frac{n}{2} \rfloor$. To prove this, observe that since $\forall \ell\in [k]$ $\mu_\ell \notin \{\mu_1,\mu_2,\dots, \mu_k\}\cup \{{\bf 0}\}$ there must be at least two girls
who are not matched to their best stable-partners in the stable matching~$\mu_\ell$. Furthermore, recall that each girl is matched to
their best stable-partner in every matching $\{\mu_1,\mu_2,\dots, \mu_k\}$ except at most one. It immediately follows that 
$k\le \lfloor \frac{n}{2} \rfloor$.
\end{proof}

Can this upper bound on the cardinality of the minimum winning coalition ever be obtained?
The answer is yes. In fact, every integer between $0$ and $\lfloor \frac{n}{2} \rfloor$
can be the cardinality of the smallest winning coalition.
\begin{theorem}\label{thm:all-integers}
For each $0\le k \le  \lfloor \frac{n}{2} \rfloor$ there exists a stable matching instance
where the minimum winning coalition has cardinality exactly $k$.
\end{theorem}
\begin{proof}
Take any $0\le k \le  \lfloor \frac{n}{2} \rfloor$. 
We construct a stable matching instance
where the minimum winning coalition has cardinality exactly $k$ as follows.
For $2k+1 \le \ell \le n$, let boy $b_\ell$ and girl $g_\ell$ rank each other
top of their preference lists -- the other rankings in their preference lists may be arbitrary. 
Thus, boy $b_\ell$ and girl $g_\ell$ must be matched together in every stable matching. 

For $1 \le \ell \le k$, let boy $b_{2\ell-1}$ rank girl $g_{2\ell-1}$ first and girl $g_{2\ell}$ second
and let boy $b_{2\ell}$ rank girl $g_{2\ell}$ first and girl $g_{2\ell-1}$ second.
In contrast, let girl $g_{2\ell-1}$ rank boy $b_{2\ell}$ first and boy $b_{2\ell-1}$ second
and let girl $g_{2\ell}$ rank boy $b_{2\ell-1}$ first and boy $b_{2\ell}$ second.
Again, all other rankings may be arbitrary.

It is then easy to verify that two possibilities arise. In any stable matching for each $1\leq \ell\leq k$ either (i) both 
boys $b_{2\ell-1}$ and $b_{2\ell}$ are matched to their best stable-partners, namely 
girls $g_{2\ell-1}$ and $g_{2\ell}$, respectively, or (ii) both 
boys $b_{2\ell-1}$ and $b_{2\ell}$ are matched to their worst stable-partners, namely 
girls $g_{2\ell}$ and $g_{2\ell-1}$, respectively.

But this implies that to obtain the girl-optimal stable matching at least one girl from the
pair $\{g_{2\ell-1}, g_{2\ell}\}$ must misreport her preferences, for each $1 \le \ell \le k$.
One girl from each of these pairs is also sufficient to output the girl-optimal stable matching.
Thus the minimum winning coalition has cardinality exactly $k$.
\end{proof}

We remark that the instances constructed in the proof of Theorem~\ref{thm:all-integers} have $2^k$ stable matchings. 
As $k$ can be as large as $\lfloor \frac{n}{2} \rfloor$, this gives a simple proof of the well known fact that the number of stable 
matchings may be exponential in the number of participants~\cite{Knu82}.

We now have all the tools required to address the main questions in this paper. 
We will first, as promised, illustrate these tools using an example.

%% file: Example-Merged.tex
\section{An Illustrative Example}\label{sec:ex}

Here we present an example to illustrate the main concepts covered in the paper.
This stable matching instance is derived from an example constructed by Irving et al.~\citep{IRL87}. 
There are eight boys and eight girls whose preference rankings are shown in Table~\ref{tab:preferences}.

\begin{table}[ht!]
\centering
\caption{A Stable Matching Instance.}

\begin{tabular}{|c|c|c|c|c|c|c|c|c|}
\hline
{\bf \backslashbox{Boy}{Rank}} & {\bf 1} & {\bf 2} & {\bf 3} & {\bf 4} &  {\bf 5} & {\bf 6} &  {\bf 7} & {\bf 8}\\
\hline\hline
${\bf b_1}$ & $g_4$ & $g_3$ & $g_8$ & $g_1$ & $g_2$ & $g_5$ & $g_7$ & $g_6$\\

${\bf b_2}$ & $g_3$ & $g_7$ & $g_5$ & $g_8$ & $g_6$ & $g_4$ & $g_1$ & $g_2$\\

${\bf b_3}$ & $g_7$ & $g_5$ & $g_8$ & $g_3$ & $g_6$ & $g_2$ & $g_1$ & $g_4$\\

${\bf b_4}$ & $g_6$ & $g_4$ & $g_2$ & $g_7$ & $g_3$ & $g_1$ & $g_5$ & $g_8$\\

${\bf b_5}$ & $g_8$ & $g_7$ & $g_1$ & $g_5$ & $g_6$ & $g_4$ & $g_3$ & $g_2$\\

${\bf b_6}$ & $g_5$ & $g_4$ & $g_7$ & $g_6$ & $g_2$ & $g_8$ & $g_3$ & $g_1$\\

${\bf b_7}$ & $g_1$ & $g_4$ & $g_5$ & $g_6$ & $g_2$ & $g_8$ & $g_3$ & $g_7$\\

${\bf b_8}$ & $g_2$ & $g_5$ & $g_4$ & $g_3$ & $g_7$ & $g_8$ & $g_1$ & $g_6$\\
\hline
\end{tabular}

\

\

\begin{tabular}{|c|c|c|c|c|c|c|c|c|}
\hline
{\bf \backslashbox{Girl}{Rank}} & {\bf 1} & {\bf 2} & {\bf 3} & {\bf 4} &  {\bf 5} & {\bf 6} &  {\bf 7} & {\bf 8}\\
\hline\hline
${\bf g_1}$ & $b_3$ & $b_1$ & $b_5$ & $b_7$ & $b_4$ & $b_2$ & $b_8$ & $b_6$\\

${\bf g_2}$ & $b_6$ & $b_1$ & $b_3$ & $b_4$ & $b_8$ & $b_7$ & $b_5$ & $b_2$\\

${\bf g_3}$ & $b_7$ & $b_4$ & $b_3$ & $b_6$ & $b_5$ & $b_1$ & $b_2$ & $b_8$\\

${\bf g_4}$ & $b_5$ & $b_3$ & $b_8$ & $b_2$ & $b_6$ & $b_1$ & $b_4$ & $b_7$\\

${\bf g_5}$ & $b_4$ & $b_1$ & $b_2$ & $b_8$ & $b_7$ & $b_3$ & $b_6$ & $b_5$\\

${\bf g_6}$ & $b_6$ & $b_2$ & $b_5$ & $b_7$ & $b_8$ & $b_4$ & $b_3$ & $b_1$\\

${\bf g_7}$ & $b_7$ & $b_8$ & $b_1$ & $b_6$ & $b_2$ & $b_3$ & $b_4$ & $b_5$\\

${\bf g_8}$ & $b_2$ & $b_6$ & $b_7$ & $b_1$ & $b_8$ & $b_3$ & $b_4$ & $b_5$\\
\hline
\end{tabular}
\label{tab:preferences}
\end{table}

This instance has 23 stable matchings. To see this, let's begin by running the deferred acceptance algorithm (Algorithm~1) to find the 
boy-optimal stable matching. Observe that all the boys have different first preferences. Thus, the boys will consecutively each propose to their first 
choice who will temporarily accept; once all the boys have proposed these matches will become permanent.
Thus, the boy-optimal matching ${\bf 1}=M_1$ simply matches each boy with his favourite girl.
To find the remaining stable matchings, we swap partners using rotations. 

We start by finding the rotations exposed at ${\bf 1}= M_1$.
To do this we simply search for the for directed cycles in the {\em (exposed) rotation graph} $H({\bf 1})$.  
Recall, this graph has eight vertices, one for each boy. The first vertex $1$ corresponds to boy $b_1$ who is matched to girl $g_4$
in $M_1$. If girl $g_4$ breaks up with $b_1$ then he will propose to his second choice, girl $g_3$.
She prefers him over her current partner $b_2$ so will accept this offer. 
Thus boy $b_1$ will gain the current partner of boy $b_2$; hence there is an arc $(1,2)$ in the rotation graph $H({\bf 1})$.
To find the outgoing arc at vertex $2$ assume that $g_3$ does break-up with boy $b_2$. He will then propose 
 his second choice, girl $g_7$.
She prefers him over her current partner $b_3$ so will accept this offer. Thus, there is an arc at $(2,3)$. 
Now $(b_3,g_7)\in M_1$ and if girl $g_7$ breaks-up with $b_3$ then he will next propose to $g_5$. She will accept as she prefers $b_3$ over her current partner
boy $b_6$. So the rotation graph contains the arc $(3, 6)$.
Next consider boy $b_4$. If his partner, girl $g_6$ breaks up with him then he will propose to girl $g_4$. She will reject this proposal
as she prefers her current partner $b_1$ over $b_4$. So $b_4$ will then propose to his third choice $g_2$ and this proposal will be accepted
as she prefers him over her current partner $b_8$. Hence $H({\bf 1})$ contains the arc $(4, 8)$.
Continuing in this fashion, the reader can verify that the rotation graph $H({\bf 1})$ is as shown in Figure~\ref{fig:rotation-graph}.

\begin{figure}[ht!]
        \centering
            \begin{tikzpicture}[scale=.7]
              \node [shape=circle, draw] (1) at (0,2) {$1$}; 
              \node [shape=circle, draw] (2) at (2,2) {$2$}; 
              \node [shape=circle, draw] (3) at (2,0) {$3$}; 
              \node [shape=circle, draw] (4) at (-4,0){$4$}; 
              \node [shape=circle, draw] (5) at (6,0) {$5$}; 
              \node [shape=circle, draw] (6) at (0,0){$6$}; 
              \node [shape=circle, draw] (7) at (4,0){$7$}; 
              \node [shape=circle, draw] (8) at (-2,0) {$8$}; 
              \draw [->, thick] (1) -- (2);
              \draw [->, thick] (2) -- (3);
              \draw [->, thick] (3) -- (6);
              \draw [->, thick] (4) -- (8);
              \draw [->, thick] (5) -- (7);
              \draw [->, thick] (6) -- (1);
              \draw [->, thick] (7) -- (3);
              \draw [->, thick] (8) -- (6);
            \end{tikzpicture}    
    \caption{The (Exposed) Rotation Graph $H({\bf 1)}$ at the Boy-Optimal Stable Matching.}\label{fig:rotation-graph}
\end{figure}
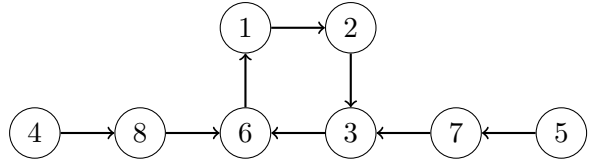

Observe that the rotation graph $H(M_1)$ contains a single cycle $\{v_1, v_2, v_3, v_6\}$. 
Consequently, there is exactly one exposed rotation, namely $\rho_1=\set{(b_1,g_4),(b_2,g_3),(b_3,g_7),(b_6,g_5)}$.
We remark that this is the unique maximal rotation for this stable matching instance.
Thus from $M_1=(g_4, g_3, g_7, g_6, g_8, g_5, g_1, g_2)$ we may create one new stable matching by performing
the rotation $\rho_1$. Specifically, we rotate the partners of the boys $\{b_1, b_2, b_3, b_6\}$.
This gives the stable matching $M_2=(g_3, g_7, g_5, g_6, g_8, g_4, g_1, g_2)$.

Similarly, for $M_2$ the rotation graph $H(M_2)$ contains a single directed cycle
$\{v_3, v_5, v_7\}$. Rotating the partners of boys $\{b_3, b_5, b_7\}$ then produces the
stable matching $M_3=(g_3, g_7, g_8, g_6, g_1, g_4, g_5, g_2)$.
The rotation graph $H(M_3)$ contains two directed cycle correspond to $\rho_{3}$ and $\rho_{4}$.
These are specified, along with all the other rotations in Table~\ref{tab:10-rotations}.

\begin{table}[ht!]
\centering
 \caption{The Set of the Rotations $\mathcal{R}$.}\label{tab:10-rotations}
\begin{tabular}{|c|c|}
\hline
{\bf Rotation} & {\bf Rotation }\\
\hline\hline
$\rho_1$ & $[(b_1,g_4),(b_2,g_3),(b_3,g_7),(b_6,g_5)]$\\
\hline
$\rho_2$ & $[(b_3,g_5),(b_5,g_8),(b_7,g_1)]$\\
\hline
$\rho_3$ & $[(b_4,g_6),(b_8,g_2),(b_7,g_5)]$\\
$\rho_4$ & $[(b_1,g_3),(b_3,g_8)]$\\
\hline
$\rho_5$ & $[(b_2,g_7),(b_8,g_5),(b_6,g_4)]$\\
$\rho_6$ & $[(b_3,g_3),(b_4,g_2)]$\\
$\rho_7$ & $[(b_1,g_8),(b_5,g_1),(b_7,g_6)]$\\
\hline
$\rho_8$ & $[(b_5,g_6),(b_8,g_4),(b_6,g_7)]$\\
$\rho_9$ & $[(b_2,g_5),(b_7,g_8),(b_4,g_3)]$\\
$\rho_{10}$ & $[(b_1,g_1),(b_3,g_2)]$\\
\hline
\end{tabular}
\end{table}

These ten rotations form the {\em rotation poset} $(\mathcal{R}, \geq)$ whose Hasse diagram is given in Figure~\ref{fig:rotation-poset}.
As shown the rotation $\rho_1$ is the unique maximal rotation, the rotation exposed at the boy-optimal stable matching ${\bf 1}=M_1$.
On the other hand, $\rho_8, \rho_9$ and $\rho_{10}$ are the minimal rotations. These rotations lead to the boy-pessimal (girl-optimal)
stable matching ${\bf 0}=M_{23}$. 

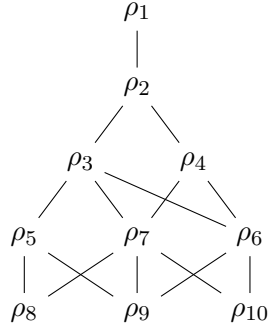
\begin{figure}[ht!]
        \centering
            \begin{tikzpicture}[scale=.5]
              \node (B-Opt) at (0,4) {$\rho_1$};
              \node (1) at (0,2) {$\rho_2$};
              \node (2-A) at (1.5,0) {$\rho_4$};
              \node (2-B) at (-1.5,0) {$\rho_3$};
              \node (3-A) at (3,-2) {$\rho_6$};
              \node (3-B) at (0,-2) {$\rho_7$};
              \node (3-C) at (-3,-2) {$\rho_5$};
              \node (4-A) at (3,-4) {$\rho_{10}$};
              \node (4-B) at (0,-4) {$\rho_9$};
              \node (4-C) at (-3,-4) {$\rho_8$};
              \draw (B-Opt) -- (1) -- (2-A) -- (3-A) -- (4-A) -- (3-B) -- (2-A);
              \draw (1) -- (2-B) -- (3-A) -- (4-B) -- (3-B) -- (2-B) -- (3-C) -- (4-C) -- (3-B);
              \draw (3-C) -- (4-B);
            \end{tikzpicture}
    
    \caption{The Hasse diagram of the Rotation Poset $(\mathcal{R}, \geq)$.}\label{fig:rotation-poset}
\end{figure}

Applying these rotations in the appropriate order allows us to generate all $23$ stable matchings given in Table~\ref{tab:23-stable-matchings}.

\begin{table}[ht!]
\centering
 \caption{The Set of Stable Matchings $\mathcal{M}$. For a matching $M$, $n(M)$ is the set of rotations 
 exposed in $M$ and $p(M)$ is the set of rotations that can precede obtaining $M$}\label{tab:23-stable-matchings}

\begin{tabular}{|c|c|c|c|c|c|c|c|c|c|c|}

\hline

{\bf $M$} & {\bf $n(M)$} & {\bf $p(M)$} & 
${\bf g_1}$ & ${\bf g_2}$ & ${\bf g_3}$ & ${\bf g_4}$ & ${\bf g_5}$ & ${\bf g_6}$ & ${\bf g_7}$ & ${\bf g_8}$\\

\hline\hline
$M_1$ & $\rho_1$ & $\emptyset$ &
$b_7$ & $b_8$ & $b_2$ & $b_1$ & $b_6$ & $b_4$ & $b_3$ & $b_5$\\

\hline

$M_2$ & $\rho_2$ & $\rho_1$ &
$b_7$ & $b_8$ & $b_1$ & $b_6$ & $b_3$ & $b_4$ & $b_2$ & $b_5$\\

\hline

$M_3$ & $\rho_3, \rho_4$ & $\rho_2$ &
$b_5$ & $b_8$ & $b_1$ & $b_6$ & $b_7$ & $b_4$ & $b_2$ & $b_3$\\

\hline

$M_4$ & $\rho_3$ & $\rho_4$ &
$b_5$ & $b_8$ & $b_3$ & $b_6$ & $b_7$ & $b_4$ & $b_2$ & $b_1$\\

$M_5$ & $\rho_4, \rho_5$ & $\rho_3$ &
$b_5$ & $b_4$ & $b_1$ & $b_6$ & $b_8$ & $b_7$ & $b_2$ & $b_3$\\

\hline

$M_6$ & $\rho_5,\rho_6,\rho_7$ & $\rho_3, \rho_4$ &
$b_5$ & $b_4$ & $b_3$ & $b_6$ & $b_8$ & $b_7$ & $b_2$ & $b_1$\\

$M_7$ & $\rho_4$ & $\rho_5$ &
$b_5$ & $b_4$ & $b_1$ & $b_8$ & $b_2$ & $b_7$ & $b_6$ & $b_3$\\

\hline

$M_8$ & $\rho_5,\rho_7$ & $\rho_6$ &
$b_5$ & $b_3$ & $b_4$ & $b_6$ & $b_8$ & $b_7$ & $b_2$ & $b_1$\\

$M_9$ & $\rho_5,\rho_6$ & $\rho_7$ &
$b_1$ & $b_4$ & $b_3$ & $b_6$ & $b_8$ & $b_5$ & $b_2$ & $b_7$\\

$M_{10}$ & $\rho_6,\rho_7$ & $\rho_4,\rho_5$ &
$b_5$ & $b_4$ & $b_3$ & $b_8$ & $b_2$ & $b_7$ & $b_6$ & $b_1$\\

\hline
$M_{11}$ & $\rho_5,\rho_{10}$ & $\rho_6,\rho_7$ &
$b_1$ & $b_3$ & $b_4$ & $b_6$ & $b_8$ & $b_5$ & $b_2$ & $b_7$\\

$M_{12}$ & $\rho_7$ & $\rho_5,\rho_6$ &
$b_5$ & $b_3$ & $b_4$ & $b_8$ & $b_2$ & $b_7$ & $b_6$ & $b_1$\\

$M_{13}$ & $\rho_6,\rho_8$ & $\rho_5,\rho_7$ &
$b_1$ & $b_4$ & $b_3$ & $b_8$ & $b_2$ & $b_5$ & $b_6$ & $b_7$\\

\hline

$M_{14}$ & $\rho_5$ & $\rho_{10}$ &
$b_3$ & $b_1$ & $b_4$ & $b_6$ & $b_8$ & $b_5$ & $b_2$ & $b_7$\\

$M_{15}$ & $\rho_8,\rho_9,\rho_{10}$ & $\rho_7,\rho_6,\rho_5$ &
$b_1$ & $b_3$ & $b_4$ & $b_8$ & $b_2$ & $b_5$ & $b_6$ & $b_7$\\

$M_{16}$ & $\rho_6$ & $\rho_8$ &
$b_1$ & $b_4$ & $b_3$ & $b_5$ & $b_2$ & $b_6$ & $b_8$ & $b_7$\\

\hline

$M_{17}$ & $\rho_8,\rho_9$ & $\rho_5,\rho_{10}$ &
$b_3$ & $b_1$ & $b_4$ & $b_8$ & $b_2$ & $b_5$ & $b_6$ & $b_7$\\

$M_{18}$ & $\rho_8,\rho_{10}$ & $\rho_9$ &
$b_1$ & $b_3$ & $b_7$ & $b_8$ & $b_4$ & $b_5$ & $b_6$ & $b_2$\\

$M_{19}$ & $\rho_9,\rho_{10}$ & $\rho_6,\rho_8$ &
$b_1$ & $b_3$ & $b_4$ & $b_5$ & $b_2$ & $b_6$ & $b_8$ & $b_7$\\

\hline

$M_{20}$ & $\rho_8$ & $\rho_9,\rho_{10}$ &
$b_3$ & $b_1$ & $b_7$ & $b_8$ & $b_4$ & $b_5$ & $b_6$ & $b_2$\\

$M_{21}$ & $\rho_9$ & $\rho_8,\rho_{10}$ &
$b_3$ & $b_1$ & $b_4$ & $b_5$ & $b_2$ & $b_6$ & $b_8$ & $b_7$\\

$M_{22}$ & $\rho_{10}$ & $\rho_8,\rho_9$ &
$b_1$ & $b_3$ & $b_7$ & $b_5$ & $b_4$ & $b_6$ & $b_8$ & $b_2$\\

\hline

$M_{23}$ & $\emptyset$ & $\rho_8,\rho_9,\rho_{10}$ &
$b_3$ & $b_1$ & $b_7$ & $b_5$ & $b_4$ & $b_6$ & $b_8$ & $b_2$\\

\hline
\end{tabular}
\end{table}

These stable matchings then form the stable matching lattice $(\mathcal{M}, \geqslant)$ whose Hasse diagram is illustrated 
in Figure~\ref{fig:sm-lattice}.
In this diagram, each edge is labelled by the rotation that transforms the upper stable matching into the lower stable matching.
For example, the rotation $\rho_9$ is exposed at $M_{15}$ and applying it produces the matching $M_{18}$; similarly, the rotation $\rho_4$ is 
exposed at $M_5$ and induces $M_6$.

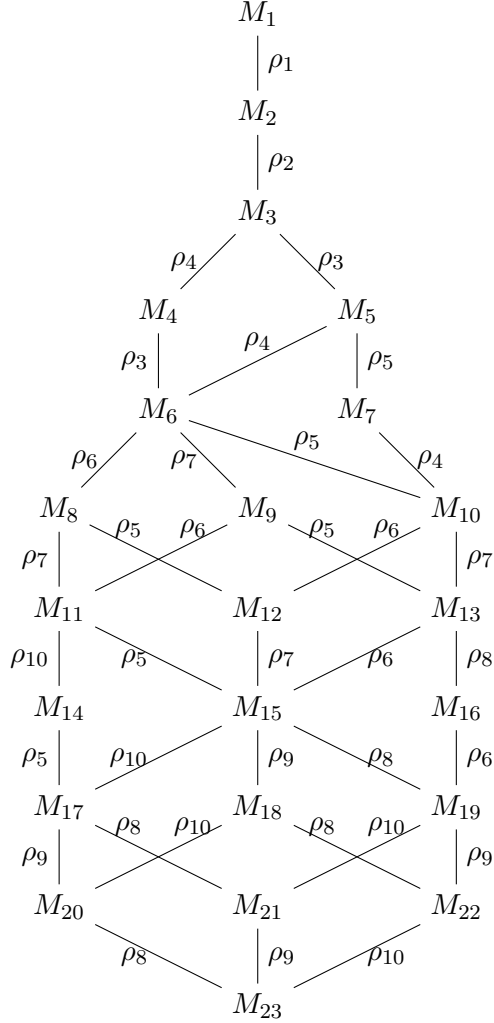
\begin{figure}[ht!]
        \centering
            \begin{tikzpicture}[scale=.66]
              \node (0) at (0,-9) {$M_{23}$};
              \node (1-A) at (4,-7) {$M_{22}$};
              \node (1-B) at (0,-7) {$M_{21}$};
              \node (1-C) at (-4,-7) {$M_{20}$};
              \node (2-A) at (4,-5) {$M_{19}$};
              \node (2-B) at (0,-5) {$M_{18}$};
              \node (2-C) at (-4,-5) {$M_{17}$};
              \node (3-A) at (4,-3) {$M_{16}$};
              \node (3-B) at (0,-3) {$M_{15}$};
              \node (3-C) at (-4,-3) {$M_{14}$};
              \node (4-A) at (4,-1) {$M_{13}$};
              \node (4-B) at (0,-1) {$M_{12}$};
              \node (4-C) at (-4,-1) {$M_{11}$};
              \node (5-A) at (4,1) {$M_{10}$};
              \node (5-B) at (0,1) {$M_9$};
              \node (5-C) at (-4,1) {$M_8$};
              \node (6-A) at (2,3) {$M_7$};
              \node (6-B) at (-2,3) {$M_6$};
              \node (7-A) at (2,5) {$M_5$};
              \node (7-B) at (-2,5) {$M_4$};
              \node (8) at (0,7) {$M_3$};
              \node (9) at (0,9) {$M_2$};
              \node (10) at (0,11) {$M_1$};

			  \draw (0,10) node [right] {$\rho_1$};
			  \draw (0,8) node [right] {$\rho_2$};

			  \draw (1,6) node [right] {$\rho_3$};
			  \draw (-1,6) node [left] {$\rho_4$};

			  \draw (-2,4) node [left] {$\rho_3$};
			  \draw (0,4) node [above] {$\rho_4$};
			  \draw (2,4) node [right] {$\rho_5$};

			  \draw (-3,2) node [left] {$\rho_6$};
			  \draw (-1,2) node [left] {$\rho_7$};
			  \draw (1,2) node [above] {$\rho_5$};
			  \draw (3,2) node [right] {$\rho_4$};

			  \draw (-4,0) node [left] {$\rho_7$};
			  \draw (-2.6,1) node [below] {$\rho_5$};
			  \draw (-1.3,1) node [below] {$\rho_6$};
			  \draw (1.3,1) node [below] {$\rho_5$};
			  \draw (2.6,1) node [below] {$\rho_6$};
			  \draw (4,0) node [right] {$\rho_7$};

			  \draw (-4,-2) node [left] {$\rho_{10}$};
			  \draw (-2,-2) node [left] {$\rho_5$};
			  \draw (0,-2) node [right] {$\rho_7$};
			  \draw (2,-2) node [right] {$\rho_6$};
			  \draw (4,-2) node [right] {$\rho_8$};

			  \draw (-4,-4) node [left] {$\rho_5$};
			  \draw (-2,-4) node [left] {$\rho_{10}$};
			  \draw (0,-4) node [right] {$\rho_9$};
			  \draw (2,-4) node [right] {$\rho_8$};
			  \draw (4,-4) node [right] {$\rho_6$};

			  \draw (-4,-6) node [left] {$\rho_9$};
			  \draw (-2.6,-5) node [below] {$\rho_8$};
			  \draw (-1.3,-5) node [below] {$\rho_{10}$};
			  \draw (1.3,-5) node [below] {$\rho_8$};
			  \draw (2.6,-5) node [below] {$\rho_{10}$};
			  \draw (4,-6) node [right] {$\rho_9$};

			  \draw (-2,-8) node [left] {$\rho_8$};
			  \draw (0,-8) node [right] {$\rho_9$};
			  \draw (2,-8) node [right] {$\rho_{10}$};

              \draw (0) -- (1-A);

              \draw (0) -- (1-B);
              \draw (0) -- (1-C);
              \draw (1-A) -- (2-A);
              \draw (1-A) -- (2-B);
              \draw (1-B) -- (2-A);
              \draw (1-B) -- (2-C);
              \draw (1-C) -- (2-B);
              \draw (1-C) -- (2-C);
              \draw (2-A) -- (3-A);
              \draw (2-A) -- (3-B);
              \draw (2-B) -- (3-B);
              \draw (2-C) -- (3-B);
              \draw (2-C) -- (3-C);
              \draw (3-A) -- (4-A);
              \draw (3-B) -- (4-A);
              \draw (3-B) -- (4-B);
              \draw (3-B) -- (4-C);
              \draw (3-C) -- (4-C);
              \draw (4-A) -- (5-A);
              \draw (4-A) -- (5-B);
              \draw (4-B) -- (5-A);
              \draw (4-B) -- (5-C);
              \draw (4-C) -- (5-B);
              \draw (4-C) -- (5-C);
              \draw (5-A) -- (6-A);
              \draw (5-A) -- (6-B);
              \draw (5-B) -- (6-B);
              \draw (5-C) -- (6-B);
              \draw (6-A) -- (7-A);
              \draw (6-B) -- (7-A);
              \draw (6-B) -- (7-B);
              \draw (7-A) -- (8);
              \draw (7-B) -- (8);
              \draw (8) -- (9);
              \draw (9) -- (10);
            \end{tikzpicture}

    \caption{The Hasse Diagram of the Stable Matching Lattice $(\mathcal{M}, \geqslant)$.}\label{fig:sm-lattice}
\end{figure}

Recall, by the {\em fundamental theorem for finite distributive lattices}, the stable matching lattice $(\mathcal{M}, \geqslant)$
has an auxiliary poset whose order ideals,
 ordered by inclusion, form $\mathcal{M}$. We claimed that this auxiliary poset is the rotation poset $\mathcal{P}=(\mathcal{R}, \geq)$.
By inspection of Figure~\ref{fig:rotation-poset} and Figure~\ref{fig:sm-lattice}, the reader 
may verify that this is indeed the case for this stable matching instance.
In particular, we can see the correspondence between minimal stable matchings in $\mathcal{M}\setminus \{{\bf 0}\}$ and minimal rotations.	
In this case there are three such minimal matchings, namely $\{M_{20},M_{21},M_{22}\}$ and three minimal rotations,
namely, $\{\rho_8,\rho_9,\rho_{10}\}$.

Finally, Theorem~\ref{thm:winning}, tells us that the minimum winning coalition has cardinality equal the number of minimal rotations. 
In this instance, these three rotations consist of the girls $\set{g_1,g_2}$, $\set{g_3,g_5,g_8}$ 
and $\set{g_4,g_6,g_7}$, respectively. A minimum winning coalition must contain exactly one element of each of these groups.
For example, the three girls $\{g_1,g_3,g_4\}$ form a minimum winning coalition and allow us to descend all the way from ${\bf 1}$ to ${\bf 0}$.
If we select only one girl we descend the lattice only as far down as $M_{17}$, $M_{18}$ or $M_{19}$. If
we select two girls we can descend only as far down as $M_{20}$, $M_{21}$, or $M_{22}$. 

%% file: Models-Merged.tex
\section{The Random Matching Model}\label{sec:random-model}

For the rest of the paper we use the {\em random matching model} which was first studied by Wilson~\cite{Wil72} and subsequently
examined in detail by Knuth, Pittel and coauthors~\cite{Knu82,Pit89,KMP90, Pit92}.
Here the preference ranking of each boy and each girl is drawn uniformly and independently from 
the symmetric group~${\bf S_n}$. Specifically, each preference ranking is
a random permutation of the set $[n]=\{1,2,\dots,n\}$.

We may now state the two main results of the paper. First, in the random matching model,
the expected cardinality of the minimum winning coalition is $O(\log n)$.
\begin{theorem}\label{main-thm}
In the random matching model, the expected cardinality of the minimum winning coalition $F$ 
is 
$$\bb{E}(|F|)=\frac{1}{2}\log(n)+O(\log \log n)$$
\end{theorem}
So the minimum winning coalition is small. Surprisingly, in sharp contrast, our second result states that
a random coalition must contain nearly {\bf every} girl if it is to form a winning coalition with high probability.
Equivalently: 
\begin{theorem}\label{main-thm-2}
In the random matching model, $\forall \varepsilon>0,\ \exists \delta(\varepsilon)>0$ such that for a random coalition $F$ of 
cardinality $(1-\varepsilon)\cdot n$ the probability that $F$ is {\bf not} a winning coalition is at least $\delta(\varepsilon)$.
\end{theorem}
To prove these results, recall Theorem~\ref{thm:winning} which states that a winning coalition $F$ must intersect
each {\em minimal rotation} in the rotation poset $(\mathcal{R}, \ge)$.
Thus, for Theorem~\ref{main-thm} it suffices to show that the expected number of minimal rotations is $O(\log n)$.
To show Theorem~\ref{main-thm-2} we must lower bound the probability that a randomly chosen coalition of girls
contains at least one girl in each minimal rotation. Our approach is to show the likelihood of a small cardinality minimal rotation is high.
In particular, we prove there is a minimal rotation containing exactly two girls with constant probability.
It immediately follows that a random coalition must contain nearly all the girls if it is to be a winning coalition with high
probability. 

\subsection{Overview of the Proofs}
So our proofs require that we study the set of minimal rotations in the random matching model.
The following two ``tricks'' will be useful in performing our analyses.
First, instead of minimal rotations we may, in fact, study the set $\mathcal{R}^{\max}$ of {\em maximal rotations},
that is the rotations that are exposed at the boy-optimal stable matching ${\bf 1}$.
This is equivalent because Theorem~\ref{thm:inverse} tells us that the inverse lattice $(\mathcal{M}, \leqslant)$
is the stable matching lattice ordered according to the preferences of the girls.
This symmetry implies that the behaviour of minimal rotations is identical to the behaviour
of maximal rotations as the maximal rotations of one lattice are the minimal rotations of the other. But why is the switch to maximal rotations from minimal rotations helpful?
Simply put, as we are using the boy proposal version of the deferred acceptance algorithm, we obtain the
boy-optimal stable matching and, consequently, it is more convenient to reason about the
rotations exposed at ${\bf 1}$, that is the maximal rotations.

Second, it will be convenient to view the deferred acceptance algorithm with random preferences in an alternative manner. 
In particular, instead of generating the preference rankings in 
advance, we may generate them dynamically.  Specifically, when a boy $b$ is selected to make a proposal he asks a girl $g$ chosen 
uniformly at random. If $b$ has already proposed to $g$ then this proposal is immediately rejected; such a 
proposal is termed {\em redundant}. Meanwhile, $g$ maintains a preference ranking only for the boys that 
have proposed to her. Thus if this is the $k$th distinct proposal made to girl $g$ then she assigns to $b$ a rank 
chosen uniformly at random among $\{1,\ldots k\}$. In particular, in the 
deferred acceptance algorithm $g$ 
accepts the proposal with probability $1/k$.  As explained by Knuth et al.~\cite{KMP90}, this process is 
equivalent to randomly generating the preference rankings independently in advance. Furthermore, recall from Theorem~\ref{thm:boy-optimal} 
that the deferred acceptance algorithm will output the boy-optimal stable matching regardless of the order of proposals.
It follows that, for the purposes of analysis, we may assume the algorithm selects the unmatched boy 
with the lowest index to make the next proposal.

So our task now is to investigate the properties of maximal rotations, that is directed cycles in the rotation graph~$H({\bf 1})$. Intuitively,
this relates to the study of directed cycles in {\em random graphs} with out-degrees exactly one. But there is one
major problem. In random graphs the choice of out-neighbour is independent for each vertex. But in the rotation graph~$H({\bf 1})$
this independence is lost. In particular, the arcs in $H({\bf 1})$ share intricate dependencies and specifically depend on who made and 
who received each proposal in obtaining the boy-optimal stable matching ${\bf 1}$. Moreover, a vertex may even have out-degree zero in 
$H({\bf 1})$. Essentially, the remainder of paper is devoted to showing that the myriad of dependencies that arise 
are collectively of small total consequence. It will then follow that the expected number or maximal rotations and the
minimum cardinality of a maximal rotation both behave in a predictable manner, similar to that of directed cycles in random graphs 
with out-degrees exactly one. Namely, the expected number of cycles is close to $\frac{\log n}{2}$ and the existence a cycle of size two with constant probability~\cite{FO90}.

Consequently, to study maximal rotations we must consider $H({\bf 1})$. 
We do this via a two-phase approach. In the {\em first phase} we calculate the boy-optimal stable matching ${\bf 1}$, without loss of generality, 
${\bf 1}=\{ (b_1,g_1), (b_2,g_2),\dots,(b_n, g_n)\}$.
This of course can be found by running the boy-proposal deferred acceptance algorithm. In the {\em second phase},
we calculate the rotation graph $H({\bf 1})$. But, as explained in Section~\ref{sec:rotation-poset} and illustrated in the example
of Section~\ref{sec:ex}, we can find the rotations by running the boy-proposal deferred acceptance algorithm longer.

In fact, to calculate (i) the expected number of maximal rotations and (ii) the probability that there is a maximum rotation
of cardinality $2$, we will not need the entire rotation graph $H({\bf 1})$ only subgraphs of it.
Moreover, the subgraphs we require will be different in each case. Consequently, the second phases required to prove
Theorem~\ref{main-thm} and Theorem~\ref{main-thm-2} will each be slightly different.
These distinct second phases will be described in detail in Section~\ref{sec:minimum-coalitions} and 
Section~\ref{sec:random-coalitions}, respectively. They both, however, share fundamental properties which
will be exploited in shortening the subsequent proofs. 

\subsection{A Technical Tool for Counters}
Before describing the two algorithms, we present a technical lemma that we will use repeated 
in analyzing the deviations that arise in their application. To formalize the lemma, we require the notion of a {\em state}. The state  
of the algorithm at any point is the record of all the (random) choices made so far: the sequence of proposals and 
the preference rankings generated by the girls. Thus we are working in the probability space  $(\Omega, P)$  of all
possible states $\Omega$ of the algorithm and the probabilities of reaching them.

We index the intermediate states of the algorithm by the number of proposals made to reach it. Let $\Omega_t$ denote the 
set of all possible states 
of the procedure after $t$ proposals. Thus $\Omega_t$ can be thought of as a partition of $\Omega$, and the 
partition $\Omega_{t+1}$ refines the partition $\Omega_t$ for any $t$. A random variable $X_t$ is {\em $\Omega_t$-measurable} if $X_t$ 
is determined by the algorithm state after $t$ proposals, that is $X$ is constant on each part of $\Omega_t$.
We say that a sequence $(X_t)_{t \geq 0}$ of random variables is a \emph{counter} if $X_t$ is  
$\Omega_t$-measurable and $X_t - X_{t-1} \in \{0,1\}$. Thus counters count the number of certain events 
occurring over the course of the algorithm. As an example, the number of successful proposals among the first $t$ proposals 
is a counter.

Our main tool is Lemma~\ref{lem:Azuma} below which is used to control large 
deviations of counters. Let $B_{k,p}$ be a random variable which follows a binomial distribution with 
parameters $k$ and $p$. 
We say that a collection of states $\mc{G}$ is {\em monotone} if for 
every state $S \not \in \mc{G}$ we have $S' \not \in \mc{G}$ for every state $S'$ that can be reached 
from $S$. For example, the collection of states in which every girl received at most one proposal is monotone. Let $\set{\mc{S}_t|t\in \bb{N}}$ be the sequence of random variables corresponding to the state of the algorithm at time $t$

\begin{lemma}\label{lem:Azuma}
Let $\mc{G}$ be a monotone collection of states and let $(X_t)_{t \geq 0}$ be a counter.\\
(a) If $\fn{P}{X_{t'+1} - X_{t'} = 1 | \mc{S}_{t'}=S_{t'}} \geq p$
	for every state $S_{t'}  \in \Omega_{t'} \cap  \mc{G}$, for any $t'\in [t,t+k]$, then, for any $\lambda \geq 0$ and any $k \geq 1$,
 \begin{equation*}
	\fn{P} { (X_{t+k} - X_t  \leq \lambda) \wedge (S_{t+k-1} \in \mc{G}) | \mc{S}_t=S_{t}} \ \leq \ \fn{P} {B_{k,p} \leq \lambda}.  \label{e:Azumadown}
\end{equation*}	
(b) If  $\fn{P}{X_{t'+1} - X_{t'} = 1 | \mc{S}_{t'}=S_{t'}} \leq p$ 
for every state $S_{t'}  \in \Omega_{t'} \cap  \mc{G}$, for any $t'\in [t,t+k]$, then, for any $\lambda \geq 0$ and any $k \geq 1$,
\begin{equation*}
 \fn{P} { (X_{t+k} - X_t  \geq  \lambda) \wedge (\mc{S}_{t+k-1} \in \mc{G}) | \mc{S}_t=S_{t}} \ \leq \ \fn{P} {B_{k,p} \geq \lambda}.  \label{e:Azumaup}
\end{equation*}
\end{lemma}	
This implies that if we can bound the probability of the counter being incremented tightly enough then the behaviour of the counter will 
be similar to the behaviour of a binomial random variable. Evidently, this will be useful because binomial random variables are much 
simpler to work with.

\begin{proof}  We prove (a) by induction on $k$. 
The base case $k=1$ is immediate. For the induction step, note that if $S_t \not \in \mc{G}$ then the left 
side of (a) is zero. Thus we may assume $S_{t} \in \mc{G}$ and hence $\fn{P}{X_{t+1} - X_{t} = 1 | S_{t}} \geq p$. 
By the induction hypothesis, we then have 
\begin{align*}
\fn{P}{X_{t+k}- X_{t+1}  \leq \lambda  -1 | (X_{t+1}- X_{t} = 1) \wedge \mc{S}_t=S_{t}} &\ \leq\ \fn{P} { B_{k-1,p} \leq \lambda -1}\\
&\mathrm{and} \\
\fn{P}{X_{t+k}- X_{t+1}  \leq \lambda | (X_{t+1}- X_{t} = 0) \wedge \mc{S}_t=S_{t}} &\ \leq\ \fn{P} { B_{k-1,p} \leq \lambda}.
\end{align*}
Since $\fn{P} { B_{k-1,p} \leq \lambda-1}\le  \fn{P} { B_{k-1,p} \leq \lambda}$, combining these three inequalities gives
\begin{align}
\fn{P}{X_{t+k}- X_{t} \leq \lambda | \mc{S}_t=S_{t}}
&\ =\  \fn{P}{X_{t+1} - X_{t} = 1 | \mc{S}_t=S_{t}}\cdot \fn{P} { B_{k-1,p} \leq \lambda -1} \nonumber \\
&\quad\qquad +\  \fn{P}{X_{t+1} - X_{t} = 0 | \mc{S}_t=S_{t}}\cdot \fn{P} { B_{k-1,p} \leq \lambda} \nonumber\\
&\ \leq\  p \fn{P} { B_{k-1,p} \leq \lambda -1} + (1-p )\fn{P} { B_{k-1,p} \leq \lambda} \label{in1}\\
&\ \leq\ \fn{P} { B_{k,p} \leq \lambda} \nonumber
\end{align}
where we obtain~(\ref{in1}) by noting the following:
\begin{align*}
\fn{P}{B_{k-1,p} \leq \lambda -1} &\leq \fn{P}{B_{k-1,p} \leq \lambda}\\
\Longrightarrow  \quad p\cdot \left[\fn{P}{B_{k-1,p} \leq \lambda -1}
- \fn{P}{B_{k-1,p} \leq \lambda}\right]&\leq 0\\
\Longrightarrow  \quad p \fn{P} { B_{k-1,p} \leq \lambda -1} + (1-p )\fn{P} { B_{k-1,p} \leq \lambda}&\leq \fn{P} { B_{k,p} \leq \lambda}.
\end{align*}
The proof of (b) is completely analogous.	
\end{proof}	
	 
In our subsequent analyses we will combine Lemma~\ref{lem:Azuma} with the following well-known Chernoff bounds that control 
deviations of $B_{k,p}$ from the mean.

\begin{lem}\label{lem:Chernoff}	 For $0 \leq \delta \leq 1$, 
\begin{align*}
\fn{P}{ B_{k,p}  \geq  (1 + \delta)pk} \ \leq\ \fn{\exp}{-\frac{\delta^{2}pk}{3}} 
\quad \mathrm{and} \quad
\fn{P}{ B_{k,p} \leq (1- \delta)pk} \ \leq\ \fn{\exp}{-\frac{\delta^{2}pk}{2}}. 
	\end{align*}
\end{lem}

%% file: Minimum-Merged.tex
\section{Minimum Winning Coalitions}\label{sec:minimum-coalitions}
In this section, we will evaluate the expected cardinality of the minimum winning coalition.
Recall, it suffices is to find the expected number of directed cycles, $\mathcal{R}^{\max}$, in the rotation graph~$H({\bf 1})$. 
To do this, it will be useful to describe the cardinality of $\mathcal{R}^{\max}$ in a more manipulable form.
Specifically, for any boy $b_i$ define a variable
$$
Z_i =
\begin{cases}
\frac{1}{|R|}& \text{if\ } b_i \text{\ is\ in\ a\ maximal\ rotation\ } R\\
0& \text{if\ } b_i \text{\ is\ not\ in\ a\ maximal\ rotation}
\end{cases}
$$
Then we obtain that:
\begin{equation*}
|\mathcal{R}^{\max}| \ =\ \sum_{R\in \mathcal{R}^{\max}} 1 \ =\ \sum_{R\in \mathcal{R}^{\max}} \sum_{(b,g)\in R} \frac{1}{|R|} \ =\ \sum_{i=1}^{n} Z_i
\end{equation*}
By linearity of expectation, the expected cardinality of the minimum winning coalition $F$ is
\begin{equation}\label{eq:expected-F}
\bb{E}(|F|) \ =\  \bb{E}(|\mathcal{R}^{\max}|) \ =\ \fn{\bb{E}}{\sum_{i=1}^{n} Z_i} \ =\ \sum_{i=1}^{n} \bb{E}(Z_i)
\end{equation}
As discussed in Section~\ref{sec:random-model}, the difficulty in computing $\bb{E}(|F|)$
is the myriad of dependencies that arise in the formation of the rotations in $\mathcal{R}^{\max}$.
Equation~\ref{eq:expected-F} is extremely useful in this regard. To quantify the dependency effects, rather than 
count expected rotations directly, it allows to focus simply on computing $\bb{E}(Z_i)$.

\subsection{Generating Maximal Rotations from the Rotation Graph }\label{s:rotation}
Ergo, our task now is to evaluate $\bb{E}(Z_i)$.
For this we study a two-phase randomized algorithm, henceforth referred to as the {\em algorithm}, 
for generating the potential maximal rotation containing a given boy. 
The first phase computes the boy-optimal stable 
matching ${\bf 1}=\{ (b_1,g_1), (b_2,g_2),\dots,(b_n, g_n)\}$.
In the {\em second phase} we use a variation of the deferred acceptance algorithm to 
generate arcs in (a subgraph of) the rotation graph and generate a random 
variable $Z$. 

The second phase starts with a randomly selected boy $i_1$ who makes uniformly random proposals until the first time he proposes 
to a girl $g_j$ who prefers him over her partner $b_j$ in the boy-optimal stable matching.
The boy $b_j$ will make the next sequence of proposals. 
The process terminates if we find a maximal rotation. Moreover, if this rotation is completed 
because girl $g_{i_1}$ receives and accepts a proposal then we have found a maximal rotation containing boy $i_1$.
In this case we also update $Z$. 
Formally, we initialize the second-phase by:
\begin{itemize}
	\item Choose $i_1$ from $\{1,2,\ldots,n\}$ uniformly at random.
	\item Initialize the potential cycle in the rotation digraph containing $i_1$ by setting $R=[i_1]$.
\end{itemize}	
Once $R=[i_1, \ldots,i_k]$ is found, we generate the arc of the rotation digraph emanating from $i_k$, as follows.
\begin{itemize}
	\item 
	Let boy $b_{i_k}$ make uniformly random proposals until the first time he proposes to a girl $g_{j}$ such 
	that $g_j$ ranks  $b_{i_k}$ higher than $b_j$. That is, $g_j$ ranks $b_{i_k}$ higher than her pessimal stable partner.
\begin{itemize}	
	\item If  $j \not \in R$ then we set $i_{k+1}=j$, $R=[i_1, \ldots,i_k, i_{k+1}]$, and recurse.
	\item If  $j \in R$ then we terminate the procedure. We set $Z = \frac{1}{|R|}$, if $j=i_1$, and $Z = 0$, otherwise.
\end{itemize}
	\item If, instead, boy $b_{i_k}$ gets rejected by all the girls then the vertex $i_k$ 
	has no-outgoing arcs in the rotation graph. Thus, $b_{i_1}$ 
	belongs to no maximal rotation, so we terminate the procedure and set $Z = 0$.
\end{itemize}

We emphasize that as the second phase runs, we do not change any assigned partnerships.
Specifically, when a girl receives a proposal we always compare her rank for the proposing boy to the rank of 
her pessimal partner, regardless of any other proposals she may have received during the second phase.
Now $Z= Z_{i_1}$ where $i_1$ was chosen uniformly at random. The next lemma is then implied 
by~\eqref{eq:expected-F} as the expectation of $Z$ is the average of the expectations of $Z_i$.

\begin{lemma}\label{lem:model}
$\bb{E}(|F|) \ = n\cdot \bb{E}(Z)$ where $Z$ is the random variable generated by the algorithm. \end{lemma}
Recall, $b_{i_1}$ is in a maximal rotation if and only if 
the rotation graph of the boy optimal stable matching has a cycle containing $b_{i_1}$.
Observe that every connected component of a directed graph in which each vertex has out-degree 1 contains 
exactly one cycle. Hence, if we find a cycle in the same connected component as $b_i$ but which does not contain him
then $b_{i_1}$ is not in a maximal rotation.
Then, since $|F|=\sum_{j=1}^n Z_{i_j}$, we get $\bb{E}(|F|) = \sum_{j=1}^n \bb{E}(Z_{i_j}) = n\cdot \bb{E}(Z)$.

\subsection{Properties of the Two-Phase Algorithm}
We now present a series of properties that arise with high enough probability during the two-phase process.
In particular, the process does not deviate too far from its expected behaviour. For example, the running time of each phase 
is not much longer than expected, no girl receives too many proposals, and no boy makes too many proposals. 
To formalize this, let $T_1$ and $T_2$ be the number of proposals made in the first and second phases, respectively, and 
let $T = T_1+T_2$. Further, let a {\em run} be a sequence of consecutive proposals made by the same boy in the same phase.
Now consider the following properties that may apply to a state:
\begin{enumerate}[I.]
	\item The algorithm has not terminated.\label{termination}
	\item If the algorithm is in the first phase then $t \leq 5n\log n$. If the algorithm is in the second phase then $T_1\leq 5n\log n$\label{T1}
	\item If the algorithm has not found a rotation yet then  $t \leq T_1 + \sqrt{n}\log^3 n $. \label{T2-minimal}
	\item Each girl has received at most $21\log n$ proposals. \label{girls}
	\item Each boy started at most $21\log n$ runs. \label{boy-runs}
	\item Each run contained at most $111\log^2n$ proposals. \label{runs}
	\item Each boy has made at most $\log^4 n$ proposals. \label{boys}
\end{enumerate}
Informally, we want the states of the algorithm to satisfy these property because they imply that:
\begin{itemize}
	\item No girl receives too many proposals compared to the others girls. Consequently, the event of having already accepted a 
	proposal in the second phase does not impact significantly the probability of accepting another proposal later.
	\item No boy makes too many proposals. Consequently, we do not need to worry about redundant proposals.
	\item The second phase is significantly shorter than the first phase. Consequently, the probability of a proposal being accepted at the start of the second phase and the probability of a proposal being accepted at the end of the second phase are very similar.
\end{itemize} 

Let $\mc{G}$ be the set of all states that satisfy properties \ref{termination} to \ref{boys}. 
We call these {\em good} states. Any state that is not good is {\em bad}.
Clearly, $\mc{G}$ is monotone. Let $G_*$ denote the 
event $\mc{S}_{T-1} \in \mc{G}$, that is, the event that the algorithm is in a good state the period before 
it terminates. Let $\overline{G_*}$ be the complement of $G_*$.

We remark that, for technical reasons, we will assume the second-phase terminates if $n\log n$ proposals are made during that phase. 
This assumption is superfluous here  by conditon~\ref{T2-minimal}, which states the second phase has at most $\sqrt{n}\log^3 n$ proposals.
However, the assumption is useful as it will allow the following lemma to also apply for the modified second-phase
algorithm that we use in Section~\ref{sec:random-coalitions}.

\begin{lemma}\label{l:good} For $n$ sufficiently large, $\fn{P}{G_*} \geq 1 - O(n^{-4})$.
\end{lemma}
\begin{proof} 
It suffices to show the probability is $O(1/n^4)$ of reaching a state $S_k \in \Omega_{k}$ such that (i) the algorithm has 
not yet terminated, (ii) $S_k $ is bad, and (iii) all the states preceding $S_k$ are good. 

Note that, for $n$ sufficiently large, conditions \ref{T1} and our bound on the length of the second phase imply that $k \leq  6n\log n$ for any such state. 
Furthermore, again for sufficiently large $n$, conditions \ref{boy-runs} and \ref{runs} together imply \ref{boys}; thus, \ref{boys} 
cannot be the only condition violated by $S_k$. Hence, it suffices to verify that the probability of reaching such a state violating 
one of the conditions \ref{T1}-\ref{runs} is small.
 
First consider condition \ref{T1}. Recall the first phase terminates when every girl has received a proposal.
So if $S_k$ violates \ref{T1} then $k \geq 5n\log n$ and at least one girl still has not received a proposal. 
By definition, each proposal is directed at girl $g$ with probability $1/n$, for each $g$.
So, by Lemma~\ref{lem:Azuma} applied to the counter $(X_{g,k})_{k \geq 0}$, where $X_{g,k}$ is the number of 
 proposals received by $g$ by the time $k$,
\begin{equation*}
\fn{P}{X_{g,k}=0}  \ \leq\ \left(1-\frac{1}{n}\right)^{k} \ \leq\  e^{-k/n} \ \leq\  \frac{1}{n^5} \, .
\end{equation*}
Thus, by the union bound, $S_k$ violates  \ref{T1} with probability at most $1/n^4$, as desired.
	
Next consider condition \ref{girls}.
If $s=\lceil 21 \log n \rceil$ then the probability that a girl $g$ received at least $s$ proposals is at most
\begin{align*}
\fn{P}{ X_{g,k} \geq s } \ \leq\  \binom{k}{s}\frac{1}{n^s} \ \leq\ \left(\frac{ek}{sn}\right)^s  \ \leq\ \left(\frac{6en\log n}{\lceil 21 \log n \rceil n}\right)^s 
	\ \leq\ \left(\frac{6e}{21}\right)^s\ \le \ n^{21\log(6e/21)} \ \leq\ \frac{1}{n^{5.3}}
\end{align*}
and so, by the union bound,  $S_k$ violates  \ref{girls} with probability at most $1/n^4.$
 
The proof of \ref{boy-runs} is similar.
Take a boy $b$. Apart from his first run (and possibly the first state of second phase), $b$ can only start a run if the girl $g$ he had been  
matched to received a proposal in the previous round. This occurs with probability at most $1/n$ conditioned on the previous 
state. Thus, analogously to the argument above, $S_k$ violates condition~\ref{boy-runs} with probability at most $1/n^4$.

Now consider \ref{runs} and set $s = 111\log^2 n$.
For sufficiently large $n$, the proposal following a good state is non-redundant with probability considerably greater than $21/22$ by \ref{boy-runs}.
Because each girl has received at most  $21\log n$ proposals by \ref{girls}, the probability that a proposal is accepted 
conditioned on the previous good state is at least $\frac{1}{22\log n}$.
By Lemma~\ref{lem:Azuma} applied to the counter $X_t$ equal to the number of proposals accepted by time~$t$,
\begin{align*}
\fn{P}{ X_k-X_{k-s}=0) \wedge (\mc{S}_{k-1} \in \mc{G})  | \mc{S}_{k-s}=S_{k-s}} \ \leq\ \left(1-\frac{1}{22\log n}\right)^s  \ \leq\ \fn{\exp}{-\frac{s}{22\log n}} \ \leq\ \frac{1}{n^{5.04}}
\end{align*}
for any good state $S_{k-s}$. In particular, for any such state $S_{k-s}$, the probability is at most $\frac{1}{5n^5\log n}$, for 
sufficiently large $n$. By the union bound taken over possible choices of $k \leq 5n\log n$, the probability that 
we reach $S_k$ violating \ref{runs} is at most $1/n^4$.

Set $\ell = \lceil \frac{1}{2}\sqrt{n}\log^3n \rceil$.	If $S_k$ violates \ref{T2-minimal} then the second  phase been running 
for at least $2\ell$ steps without finding a cycle, and we have previously reached 
a state $S \in \Omega_{k-\ell}$ such the second phase contained at least $\frac{\ell}{111\log^2n}$ runs before $S$ due to \ref{runs}. Therefore,  
the potential cycle $R$ generated in $S$ may contain at least $\frac{\ell}{111\log^2n}$ elements. In each subsequent step starting in a 
good state, the probability that a non-redundant proposal is made to a girl $g_i$ with $i \in R$ is at least $\frac{\ell}{n\log^3n}$. 
Further,  such a proposal is then accepted, terminating the process, with probability at least $\frac{1}{21\log n}$. However, to 
reach a state $S_k$ the algorithm must continue for at least $\ell$ more steps.
By Lemma~\ref{lem:Azuma}, the probability that this happens starting with any given state $S$ as above is at 
most 
\begin{equation*}
\fn{}{1-\frac{\ell}{21n\log^4n}}^\ell \ \leq\  \fn{\exp}{-\frac{\ell^2}{21n\log^4n}} \ \leq\  \frac{1}{n^6}
\end{equation*}
and thus the probability of reaching $S_k$ violating \ref{T2-minimal} is at most $1/n^4$.
\end{proof}
So, with high probability, we are in a good state the period before the algorithm terminates.
It follows that the magnitude of the expected number of maximal rotations can be evaluated
by consideration of good states.

\subsection{The Expected Cardinality of the Minimum Winning Coalition}
Now, to calculate the expected number of maximal rotations we must analyze in more detail
the second phase of the algorithm.
In particular, this section is devoted to the proof of the following lemma.
\begin{lem}\label{l:main1}
	Let $\mc{S}_*=\mc{S}_{T_1}$ be the terminal state of the first phase. If $\fn{P}{\overline{G_*}|\mc{S}_*=S_*}~\leq~\frac{1}{n^3}$ 
	then $$\fn{\bb{E}}{Z|\mc{S}_*=S_*} = \frac{\log n}{2n} + \fn{O}{\frac{\log\log n}{n}}.$$
\end{lem}

Before embarking on the proof of Lemma~\ref{l:main1}, we remark that that our first main result, Theorem~\ref{main-thm}, 
readily follows from it via Lemmas~\ref{lem:model} and~\ref{l:good}. 
It is also worth noting that \ref{T2-minimal} implies that the second phase has at most $\sqrt{n}\log^3 n$ proposals when $G_*$ occurs, due to the fact that we 
stop once we find our first cycle.
\begin{proof}[Proof of Theorem~\ref{main-thm} (modulo Lemma~\ref{l:main1})]
	Let $\mc{G}_{**}$ denote the set of the terminal states $S_*$ of the first phase of the algorithm 
	satisfying $\fn{P}{\overline{G_*}|\mc{S}_*=S_*} \leq \frac{1}{n^3}$. Then $\fn{P}{\mc{S}_* \not \in \mc{G}_{**}} = O(1/n)$ by Lemma~\ref{l:good}. 
	Since $0\le Z\le 1$, by Lemmas~\ref{lem:model} and~\ref{l:main1} we have
	\begin{align*}
	\bb{E}(|F|) 
		=&\  n \bb{E}(Z) \\
	=&\ \ n \bb{E}(Z | \mc{S}_* \in \mc{G}_{**}) (1-\fn{P}{\mc{S}_* \not \in \mc{G}_{**}}) + n\bb{E}(Z | \mc{S}_* \not \in \mc{G}_{**})\fn{P}{\mc{S}_* \not \in \mc{G}_{**}} \\ 
	=&\ \ n \bb{E}(Z | \mc{S}_* \in \mc{G}_{**})+  n\fn{P}{\mc{S}_* \not \in \mc{G}_{**}} \left(\bb{E}(Z | \mc{S}_* \not \in \mc{G}_{**}) - \bb{E}(Z | \mc{S}_* \in \mc{G}_{**})\right) \\ 
	=&\ \ \frac{1}{2}{\log n} + \fn{O}{\log\log n} + O(1) \qedhere
	\end{align*}
\end{proof}	
So let's prove Lemma~\ref{l:main1}. For the remainder of the section we fix $\mc{S}_*=S_*$ satisfying the conditions of 
the lemma.  Let $\rho_i$ be the number of non-redundant proposals received by girl $g_i$ in $S_*$. Set 
$\rho = \frac{1}{n}\sum_{i=1}^{n}\frac{1}{\rho_i+1}$.
As $S_*$ is good, we have $\rho_i \leq 21\log n$ for every girl $g_i$; so, $\rho \geq \frac{1}{22\log n}$.
We evaluate $\fn{\bb{E}}{Z}$ separately for every choice of initial vertex $i_1$ of $R$ in the following lemma:

\begin{lem}\label{l:main2} For every $1 \leq i \leq n$ we have
	$$\fn{\bb{E}}{Z|\mc{S}_*=S_* \wedge (i_1=i)} = \frac{1}{n\rho(\rho_{i}+1)}\left(\frac{1}{2} \log n + \fn{O}{\log\log n}\right).$$
\end{lem}

Since $\frac{1}{n}\sum_{i=1}^n \frac{1}{n\rho(\rho_i+1)}=\frac{1}{n}\frac{1}{n\rho}\sum_{i=1}^n \frac{1}{\rho_i+1}=\frac{\rho}{n\rho}=\frac{1}{n}$, this 
lemma implies Lemma~\ref{l:main1}. To prove Lemma~\ref{l:main2}, we may assume that we reached the state $S_*$, 
chose $i_1=i$, and that the probabilities of  the subsequent events are scaled accordingly. Note that, by Lemma~\ref{l:good}, we 
have $P(G_*) \geq 1-1/n^2$. We relabel the  states of our process $(\mc{S}_0=S_*, \mc{S}_1,\ldots,\mc{S}_{t},\ldots)$, so that $\mc{S}_t$ is the 
state attained after $t$ proposals have been made in the second phase.
Let $R_t$ denote the (random) set $R$ generated in the state $\mc{S}_t$. Let $X_t =|R_t|$ be the associated counter. First we 
show that any proposal made after a good state increases $X_t$ with probability close to $\rho$.

\begin{lemma}\label{lem:acceptance_prob} 
For any good state $S_{t}$, we have: 
$$\fn{P}{X_{t+1} - X_{t}=1 | \mc{S}_t=S_{t}} \in \left[\rho - n^{-1/3}, \rho\right].$$
\end{lemma}

\begin{proof}
	Let $\rho_{i,t'}$ be the number of non-redundant proposals received by girl $g_i$ in a state $S_{t'}$ 
	preceding $S_t$, and set $$\rho(t') = \frac{1}{n}\sum_{i \not \in R_{t'}}\frac{1}{\rho_{i,t'}+1}.$$ 
	Then $\rho(0) \geq \rho-1/n$ and $\rho(t'+1) \geq \rho(t')-1/n$, for every $0 \leq t' \leq t-1$, as $R_{t'}$ increases by at most 
	one vertex in any step.
	
	Let $b$ be the boy making the proposal following $S_{t}$, and let $B$ be the set of girls $b$ has already 
	proposed to. Then $|B| \leq \log^4 n$, as $S_{t}$ is good. The probability that the next proposal is accepted 
	by a girl not in $R_{t}$, thus increasing $X_t$, is then
	$$\frac{1}{n}\sum_{i \not \in R_{t} \cup B} \frac{1}{\rho_{i,t}+1} \leq \rho,$$
	immediately implying the upper bound.
	On the other hand, this probability is  lower bounded by 
	\begin{align*}
	\rho(t) - \frac{|B|}{n} \ \geq\ \rho - \frac{t+1}{n}-\frac{\log^4 n}{n} \ \geq\ \rho - n^{-1/3}
	\end{align*}
	where the last bound holds as $S_{t}$ is good, and so $t \leq \sqrt{n}\log^3 n$. 
\end{proof}

\begin{lemma}\label{lem:bounds}
\begin{enumerate}
\item []
\item For $t \geq \log^5 n, \quad \fn{P}{ \left(|X_t - \rho t| \geq  \frac{\rho t}{\log n} \right) \wedge (\mc{S}_{t-1} \in \mc{G})} \leq \frac{1}{n^2}$. 
\item For $t \geq \frac{400\log\log n}{\rho}, \quad \fn{P}{ \left( X_t \leq \frac{1}{2}\rho t\right) \wedge (\mc{S}_{t-1} \in \mc{G})} \leq \frac{1}{2\log^5 n}$.
\end{enumerate}
\end{lemma}

\begin{proof} 
Let $\delta \geq \log^{-1}n$. Recall that $\rho \geq \frac{1}{22}\log^{-1}n$, and so 
\begin{align*} \rho -  n^{-1/3} \ \geq\ \left(1 - \frac{\delta}{2}\right)\rho \ \geq \ \frac{1}{25}\log^{-1}n.
\end{align*} 
Combining  Lemma~\ref{lem:Chernoff} with Lemma~\ref{lem:Azuma} (where $t=0$ and $k=t$) and Lemma~\ref{lem:acceptance_prob} gives
\begin{align}
\fn{P}{ X_t \leq (1 -\delta) \rho t \wedge (\mc{S}_{t-1} \in \mc{G})} 
  \leq &\ \	\fn{P}{ B_{t,\rho - n^{-1/3}} \leq  (1 -\delta ) \rho t } \nonumber\\
  \leq &\ \	\fn{P}{ B_{t,\rho - n^{-1/3}} \leq  \left(1 -\frac{\delta}{2} \right)(\rho - n^{-1/3}) t } \nonumber \\ 
  \leq &\ \	\fn{\exp}{-\frac{1}{2}\left(\frac{\delta}{2}\right)^2(\rho -  n^{-1/3})t}.  \label{eq:rho-t}
\end{align}	
If $\delta=\log^{-1} n$ and $t \geq \log^5n$ then (\ref{eq:rho-t}) is upper bounded by $\fn{\exp}{-\frac{\log^2n}{200}} < \frac{1}{2n^2}$.
Meanwhile for  $\delta=1/2$ and $t \geq \frac{C\log\log n}{\rho} $,  the last term of (\ref{eq:rho-t}) can instead be upper bounded 
by $\log^{-\frac{C}{64}} n$. This proves the stated bounds on lower deviation.
	
The inequality
$$\fn{P}{\left( X_t \geq \left(1 +  \frac{1}{\log n}\right){\rho t} \right) \wedge (\mc{S}_{t-1} \in \mc{G})} \ \leq\  \frac{1}{2n^2}$$
for $t \geq \log^5 n$ is derived in the same manner.
\end{proof}

We also need the following two easy lemmas.

\begin{lemma}\label{lem:t2lower} 
We have $\fn{P}{T_2 \leq \sqrt{n}/ \log n} \leq 1/\log^2 n$.
\end{lemma}

\begin{proof}
Note that $X_t \leq t+1$, and so  for any $t \leq \sqrt{n}\log^{-1}n -1$,  the probability that the next proposal is directed 
to a girl with index in $R_t$ is at most $\frac{1}{\sqrt{n}\log n}$. Therefore, the probability that the second phase 
terminates after exactly $t$ proposals is at most $\frac{1}{\sqrt{n}\log n}$ for every such $t$. The lemma follows by 
applying the union bound. 
\end{proof}

Let $\{b_j\}_{j \in J}$ be the set of boys who have proposed to $g_i$  by the end of the first phase.

\begin{lemma}\label{lem:j} 
The probability that at least one of the first $\sqrt{n}\log^3 n$ proposals of the second phase is directed to a girl $g_j$ with $j \in J$ is at 
most $n^{-1/3}$.
\end{lemma}

\begin{proof}
As $S_*$ is 
good, we have $|J| \leq 21\log n$. Thus, this lemma follows by applying the union bound analogously to Lemma~\ref{lem:t2lower}. We omit the details.
\end{proof}

\begin{proof}[Proof of Lemma~\ref{l:main2}]
Let's begin by proving the lower bound.
Let $\mc{L}_t$ denote the collection of states $S_t$ such that 
\begin{itemize}
	\item  $\log^5 n \leq  t \leq \sqrt{n}/\log n$
	\item $S_t \in \mc{G}$, in particular the algorithm has not yet terminated,
	\item $X_t \leq \left(1+\frac{1}{\log n}\right) \rho t,$
	\item every girl $g_j$ with $j \in J$ received no proposal in the second phase so far.
\end{itemize}	
It follows from Lemmas~\ref{lem:bounds},~\ref{lem:t2lower} and~\ref{lem:j} that $P(\mc{S}_t \not \in \mc{L}_t) \leq \log^{-1} n$, 
for any  $\log^5 n \leq  t \leq \sqrt{n}\log^{-1}n$.
As any state $S_t\in \mc{L}_t$ is good and satisfies  $t \geq \log^5 n$, the boy $b_i$ has already finished the run which started the second 
phase. Moreover, no other boy who has previously proposed to $g_i$ has lost his partner and had an opportunity to make a proposal.  Thus, if the 
next proposal is directed at $g_i$, which happens with probability $1/n$, it is non redundant.  Such a proposal is accepted with 
probability $\frac{1}{\rho_i+1}$. In such a case, the algorithm terminates and outputs $Z=1/X_t$. By Lemma~\ref{lem:acceptance_prob},
	considering only the contributions of outcomes when the process terminates immediately following a state in $\mc{L}_t$ we get 
	the following lower bound on the expected value of $Z$.  
\begin{align*}
	\fn{\bb{E}}{Z|\mc{S}_*=S_* \wedge (i_1=i)} 
	\geq &\ \left(1 - \frac{1}{\log^2 n} \right) \frac{1}{n (\rho_i+1)} 
	\sum_{t=\log^5 n}^{\frac{\sqrt{n}}{\log n}} \frac{1}{\left(1+\frac{1}{\log n} \right) \rho t} \\
	= &\ \left(1 - \frac{1}{\log n} \right)\frac{1}{n \rho(\rho_i+1)} \sum_{t=\log^5 n}^{\frac{\sqrt{n}}{\log n}} \frac{1}{t} \\
	\geq &\ \left(1 - \frac{2}{\log n} \right)\frac{1}{n \rho(\rho_i+1)} \left( \log\left(\frac{\sqrt{n}}{\log n}\right) - \log(\log^5 n) -O(1) \right) \\
	=&\ \frac{1}{n \rho(\rho_i+1)} \left( \frac{1}{2}\log{n} - O(\log\log n) \right) 
	\end{align*}

Next we prove the upper bound. Let $\mc{U}_t$ denote the collection of states $S_t$ such that 
\begin{itemize}
	\item  $400\rho^{-1}\log\log n \leq t \leq \sqrt{n}\log^3n$
	\item $S_t \in \mc{G}$ or the algorithm has terminated,
	\item $X_t \geq \frac{1}{2}\rho t$.
	\item $X_t \geq \left(1-\frac{1}{\log n}\right) \rho t,$ if $t \geq \log^5 n$,
\end{itemize}

It follows from Lemma~\ref{lem:bounds} that:
$$\begin{cases}
P(\mc{S}_t \not \in \mc{U}_t) \leq \log^{-5}n\text{ for }
400\rho^{-1}\log\log n \leq t \leq \log^5 n \\ 
P(\mc{S}_t \not \in \mc{U}_t) \leq \frac{1}{n} \text{for }
\log^5 n \leq t \leq \sqrt{n}\log^3n
\end{cases}$$
Noting that the process terminates and outputs $Z=1/X_t$ immediately 
following any given state $S_t$ with probability at most  $\frac{1}{n(\rho_i+1)}$, 
we obtain  the desired upper bound on the expected value of $Z$, as follows:  
 \begin{align*}
 \fn{\bb{E}}{Z|\mc{S}_*=S_* \wedge (i_1=i)} 
 \leq &\ \frac{1}{n(\rho_i+1)} \sum_{t \leq 400\rho^{-1}\log\log n} 1 
 	+ \frac{1}{n(\rho_i+1)}\sum_{t = 400 \rho^{-1}\log\log n}^{\log^5 n} \left(\frac{1}{\log^5 n} + \frac{2}{\rho t} \right) \\
   &\qquad + \frac{1}{n(\rho_i+1)}\sum_{t = \log^5 n}^{\sqrt{n}\log^3n}\left(\frac{1}{n} + \frac{1}{\left(1-\frac{1}{\log n} \right) \rho t} \right)
  + \fn{P}{T_2 \geq \sqrt{n}\log^3n} \\
  = &\ \frac{1}{n(\rho_i+1)}\sum_{t \leq \sqrt{n}\log^3n}\frac{1}{\rho t} 
  	+ O\left(\frac{\log\log n}{n\rho(\rho_i+1)}\right) \\ 
	= &\ \frac{1}{n \rho(\rho_i+1)} \left( \frac{1}{2}\log{n} + O(\log\log n) \right).
 \end{align*}
This complete the proof of Lemma~\ref{l:main2} and thus of Lemma~\ref{l:main1}. Our first main result, Theorem~\ref{main-thm}, immediately follows.
\end{proof}

%% file: Random-Merged.tex
\section{Random Winning Coalitions}\label{sec:random-coalitions}

In this section, we consider the case where the girls in the coalition are themselves randomly selected.
Our task now is to prove that almost every girl must be selected if we wish to obtain a winning coalition asymptotically almost surely.
To do this, it will suffice to prove that there is a maximal rotation of cardinality two with constant probability.

\subsection{Generating Maximal Rotations from the Rotation Graph}
Let $Z'$ be a random variable counting the number of maximal rotations of cardinality two.
Again, to analyze $Z'$ we use a two-phase algorithm.
The {\em first phase} is the same as before. We simply generate the boy-optimal stable matching
${\bf 1}=\{ (b_1,g_1), (b_2,g_2),\dots,(b_n, g_n)\}$.
But the {\em second phase} is slightly different. 
Previously we had to evaluate the expected number of maximal rotations and, to achieve that, it sufficed to end the second phase
once we had found one rotation. Now, because we are interested in maximal rotations of cardinality two we will extend the 
second phase and terminate only when and if we find rotation of cardinality two.

So now in the second phase we use the following algorithm  to generate the random variable $Z'$, initialized at $0$:
\begin{itemize}
	\item Choose $i_1$ from $\{1,2,\ldots,n\}$ uniformly at random.
	\item Initialize the set of indices of boys who have made proposals in the second phase with ${\scr{I}}=\set{i_1}$. 
	\item Set ${\tt tar}=\infty$. 
\end{itemize}	
For motivation, at any step, girl $g_{\tt tar}$ can be viewed as the target girl. If she accepts the next proposal then this will
complete a rotation of cardinality two. Observe that we intitialize ${\tt tar}=\infty$ as it is impossible to complete a rotation in the fist step.

To complete the description of the second-phase, assume we have ${\scr{I}}=\set{i_1, \ldots,i_k}$. 
If $k<\frac{n}{2}$ {\em and} less than $n\log n$ proposals in total have been made then
we generate the next arc of the rotation digraph starting at $i_k$, as follows:
\begin{itemize}
	\item Let boy $b_{i_k}$ make uniformly random proposals until the first time he proposes to a girl $g_{j}$ such 
	that $g_j$ ranks  $b_{i_k}$ higher than $b_j$.
\begin{itemize}	
    \item If $j={\tt tar}$ then increment $Z'$ by $1$. Recurse.
	\item If $j \in {\scr{I}}\setminus\{{\tt tar}\}$ then pick $i_{k+1}$ from $\{1,2,\ldots,n\}\setminus {\scr{I}}$ uniformly at 
	random. Set ${\scr{I}}=\set{i_1, \ldots,i_k, i_{k+1}}$ and ${\tt tar}=\infty$. Recurse.
	\item If $j \not \in {\scr{I}}$ then set $i_{k+1}=j$, ${\tt tar}=i_k$, ${\scr{I}}=\set{i_1, \ldots,i_k, i_{k+1}}$. Recurse.
\end{itemize}
	\item If, instead, boy $b_{i_k}$ gets rejected by all the girls then return $Z'=0$
\end{itemize}

\begin{lemma}\label{lem:model-2} The probability of the existence of a maximal rotation of size two
is lower bounded by $P(Z'\geq 1)$.
\end{lemma}
\begin{proof}
Observe that $Z'$ is only incremented when we find a pair $(i,j)$ such that 
the next girls to accept proposals from $b_i$ and $b_j$, respectively, are $g_j$ and $g_i$.
This implies that $Z'\geq 1$ can only arise when there is a maximal rotation of cardinality $2$.
\end{proof}
Therefore, our aim is to prove that $P(Z'\geq 1) = \Omega(1)$, where $Z'$ is the random variable 
generated by the algorithm.

\subsection{Bounding the Number of Proposals}
Our objective now is to show that the behaviour of this new two-phase algorithm
does not deviate too much from its expected behaviour. Specifically, we show it satisfies a series of properties
with sufficiently high probability.
As before, let $T_1$ and $T_2$ be the number of proposals made in the first and second phases, respectively, and let $T = T_1+T_2$. 
Properties~\ref{termination} to \ref{boys} are as defined in Section~\ref{sec:minimum-coalitions}.
But now we require several more properties. To describe these,
let $p_{\mc{S}_t}$ denote the probability of the next proposal being accepted when in state $\mc{S}_t$. 
We are interested in the following five properties that may apply to a state in the second phase:
\begin{enumerate}[I.]
		\setcounter{enumi}{7}
		\item $t \geq \frac{1}{2} n\log n$ \label{T1-upper}
		\item No more than $n^{\frac{9}{10}}$ girls have received less than $\frac{1}{4}\log n$ proposals. \label{irregular}
		\item No more than $\sqrt{n}$ girls have received a redundant proposal.\label{redundant}
		\item $\set{p_{\tau}|T_1\leq \tau\leq t}\subseteq \clInt{\frac{1}{22\log n}}{\frac{5}{\log n}}$ \label{rho}
		\item $T_2 \geq \frac{1}{20} n\log n$ \label{T2-Random}
\end{enumerate}
Let $\mc{G}'$ be the set of all good states in the second phase satisfying these conditions.
Like $\mc{G}$, $\mc{G}'$ is monotone. 
Let $G^*$ denote the 
event $\mc{S}_{T-1} \in \mc{G}'$, that is the event that  the algorithm in a good state satisfying these conditions the period before it terminates.

\begin{lemma}\label{good-events-2}
For $n$ sufficiently large, $\fn{P}{G^*} \geq 1 - o(1)$.
\end{lemma}
\begin{proof}
Given the algorithm has not terminated, properties \ref{T1} to \ref{boys} hold with high probability by the
same argument as in Lemma~\ref{l:good}.
Therefore, it is enough to show that properties \ref{T1-upper} to \ref{T2-Random} hold almost surely conditioned 
on \ref{T1} to \ref{boys}.

Recall $T_1$ is the number of proposals until each girl receives at least one proposal. Thus $T_1$ is just the random variable for a coupon 
collector's problem. Let $t^i$ be the number of proposals needed to collect the $i^{th}$ coupon after the first $i-1$ coupons have already been collected.
So the $t^i$ are independent geometric random variables with parameters $\frac{n-(i-1)}{n}$ which sum to $T_1$. Thus the expectation of $T_1$ is
$\sum_{i=1}^n \bb{E}(t^i) = \sum_{i=1}^{n}\frac{n}{n-(i-1)} =  n\log n+\smallo{n\log n}$ 
with variance  
$\sum_{i=1}^n \mathrm{Var}(t^i) =\sum_{i=1}^n \left( \left(\frac{n}{(n-(i-1))}\right)^2-\frac{n}{(n-i-1)}\right)=n^2 \sum_{j=1}^n \frac{1}{j^2}-n\sum_{j=1}^n \frac{1}{j}$. 
Consequently, the variance is bounded above by $\frac{\pi^2}{6}\cdot n^2$. Applying Chebyshev's inequality then gives:
$$
\fn{P}{T_1\leq \frac{1}{2}n\log n}
\ \leq\ \fn{P}{|T_1-\bb{E}(T_1)
\ \geq\ \frac{1}{3}n\log n}
\ \leq\ \frac{\pi^2}{6\left(\frac{1}{3}\log n\right)^2}
\ =\ \frac{3\pi^2}{2 \log^2 n}
$$
This proves \ref{T1-upper} occurs almost surely as $t\geq T_1$ in the second phase.

Let $X_i$ be the indicator of whether girl $i$ has received at most $\frac{1}{4}\log n$ proposals 
after $t\geq \frac{1}{2}n\log n$ total proposals. Let $Y_i$ be the number of different proposals she has received in 
that amount of time. Then, by Markov's inequality:
$$
\fn{P}{\sum_{i=1}^n X_i\geq n^{\frac{9}{10}}}
\ \leq\  \frac{\bb{E}(\sum_{i=1}^n X_i)}{n^{\frac{9}{10}}}
\ =\  n^{\frac{1}{10}}\cdot P(X_1=1)
\ = \ n^{\frac{1}{10}}\cdot \fn{P}{Y_i\leq \frac{1}{4}\log n}
$$
Since $Y_i$ is binomial and has expectation greater than $\frac{1}{2}\log n$, applying a Chernoff bound gives:
$$\fn{P}{Y_i\leq \left(1-\frac{1}{2}\right)\frac{1}{2}\log n} \ \leq \ e^{-\frac{1}{8}\log n} \ =\ n^{-\frac{1}{8}}$$
Thus,
$$\fn{P}{\sum_{i=1}^n X_i\geq n^{\frac{9}{10}}}
\ \leq\ n^{\frac{1}{10}-\frac{1}{8}}
\ =\ n^{\frac{-1}{40}}
$$
This proves that \ref{irregular} occurs almost surely.

Next let $R$ denote the number of redundant proposals made in the first phase. 
We know that, throughout the first phase, the set of girls that any boy has proposed to has cardinality at most $\log^4 n$.
Thus, the probability that any proposal is redundant is bounded above by $\frac{\log^4 n}{n}$. 
This implies that the expected number of redundant proposals is at most:
$$\frac{\log^4 n}{n}\cdot 5 n\log n \ = \ 5\log^5 n $$
In particular, let $R_G$ be the number of girls who have received redundant proposals. Clearly $R_G\leq R$. Thus, by 
Markov's inequality:
$$ \fn{P}{R_G\geq \sqrt{n}}\ \leq \ \fn{P}{R\geq \sqrt{n}}\ \leq \ \frac{5\log^5 n}{\sqrt{n}}$$
This proves that \ref{redundant} occurs almost surely.

Now let $X$ be the set of girls who have received at least $\frac{1}{4}\log n$ distinct proposals. 
By \ref{T1-upper} and \ref{irregular}, no more than $n^{\frac{9}{10}}$ girls have received less than $\frac{1}{4}\log n$ 
proposals by the end of the first phase.
By \ref{redundant} no more than $\sqrt{n}$ of the remaining girls have received redundant proposals.
It follows that $|X|\geq n-\sqrt{n}-n^{\frac{9}{10}}$. 
Recall, $p_{\mc{S}_t}$ is the probability the next proposal being accepted and
let $\rho_i$ be the number of proposals received by $g_i$ at the end of the first phase. So, when \ref{T1-upper} to \ref{redundant} occur:

$$
p_{\mc{S}_t} \ \leq\ \frac{1}{n}\sum_{g\in G} \frac{1}{\rho_g+1}
\ \leq\ \frac{1}{n}\sum_{g\in X} \frac{1}{\rho_g+1}+\frac{\sqrt{n}+n^{\frac{9}{10}}}{n}
\ \leq\ \frac{n-n^{\frac{9}{10}}-\sqrt{n}}{n}\cdot \frac{1}{\frac{1}{4}\log n+1}+n^{-\frac{1}{2}}+n^{\frac{-1}{10}}
\ \leq\ \frac{5}{\log n}
$$
Similarly, by \ref{boys}, at least $n-\log^4 n$ girls have not received a proposal from the proposing boy. 
This, along with \ref{girls}, implies that when \ref{T1-upper} to \ref{redundant} occur $p_{\mc{S}_t}\geq \frac{n-\log^4 n}{n}\cdot \frac{1}{21 \log n+1}\geq \frac{1}{22\log n}$.
Hence, \ref{rho} occurs almost surely.

Finally, to show the last property, we modify the proof of Lemma~\ref{lem:bounds} using the bounds obtained from \ref{rho}. 
Indeed, if $t\leq \frac{n}{2}$ then the algorithm has not terminated. For $t\geq \frac{n}{2}$, denoting $X_t$ to be 
the number of accepted proposals in the second phase, applying Lemma~\ref{lem:Chernoff} with Lemma~\ref{lem:Azuma} 
for $t\leq \frac{1}{20}n\log n$ gives:
$$
\fn{P}{ X_{t} \geq \frac{n}{2}} 
\ \leq \ \fn{P}{ X_{t} \geq \frac{10}{\log n}t}
\ \leq \ \fn{P}{ B_{t,\frac{5}{\log n}} \geq  \frac{10}{\log n} t } 
\ \leq \ \fn{\exp}{-\frac{t\cdot \frac{5}{\log n}}{3}}
\ =\ o(1)  
$$	
Thus \ref{T2-Random} occurs almost surely.
\end{proof}

\subsection{Bounding the Probability of Missing a Rotation}
We can complete the proof of our second main result in two steps.
First, we show that the probability of a maximal rotation of cardinality two existing is at least a constant.
The second step is then easy. If there is a maximal rotation of cardinality two then a random coalition of cardinality at most
$(1-\epsilon)\cdot n$ will not be a winning coalition with constant probability. 

\begin{theorem}\label{thm:small-rotation}
Let $Z'_t=1$ when the $t^{th}$ proposal of the second phases closes a rotation of size $2$ and let $Z'=\sum_{t=1}^{T_2} Z'_t$. 
Then $\fn{P}{Z'\geq 1}=\Omega(1)$.
\end{theorem}
\begin{proof} 
The $t^{th}$ proposal of the second phase 
ends a rotation of size 2 if the following five events occur:
\begin{itemize}
    \item $E_1^t$: ``The second phase has not ended {\em and} we are not in the first run of the second phase.''
    \item $E_2^t$: ``The last proposal of the previous run was to a girl who hadn't yet accepted a proposal.''
    \item $E_3^t$: ``The proposal is made to the optimal partner of the boy who started the previous run.''
    \item $E_4^t$: ``The proposal is not redundant.''
    \item $E_5^t$: ``The proposal is accepted.''
\end{itemize}
If $111\log^2 n\leq t\leq \frac{1}{20} n \log n$, by \ref{boy-runs} and by \ref{T2-Random} we have $P(E_1^t|\mc{S}_t\in \mc{G}')=1$.

Denote $E_{(a,b)}^t=\bigwedge_{i=a}^b E_i^t$ and let $G_t$ be the set of girls who have accepted proposals in the second phase. 
By \ref{rho} the probability of a proposal being accepted is at most $\frac{5}{\log n}$. On the other hand, 
by \ref{girls} the probability of a proposal being accepted by a girl who isn't in $G_t$ is at 
least $\frac{1}{n}\cdot \frac{n}{2}\cdot \frac{1}{22\log n}=\frac{1}{44\log n}$, since at least $\frac{n}{2}$ girls are 
not in $G_t$. 
Thus, the probability of the proposal being accepted by a girl who isn't in $G_t$, given that it was accepted, is at least:
$$\frac{\frac{1}{44\log n}}{\frac{5}{\log n}} \ =\ \frac{1}{220}$$
Therefore, $P(E_2^t|E_1^t\wedge \mc{S}_t\in \mc{G}')\geq \frac{1}{220}$.
Now,  clearly $P(E_3^t|E_{(1,2)}^t\wedge \mc{S}_t\in \mc{G}')=\frac{1}{n}$. 
Moreover, by \ref{boys}, $P(E_4^t|E_{(1,3)}^t\wedge \mc{S}_t\in \mc{G}')\geq 1-\frac{\log^4 n}{n}$ 
and by \ref{girls}: \[P(E_5^t|E_{(1,4)}^t\wedge \mc{S}_t\in \mc{G}')\geq \frac{1}{22\log n}\]

Recall by~\ref{runs}, the first run contains at most $111\log^2 n$ proposals.
We may conclude that for large enough $n$ and $t$ such that $111\log^2 n \leq t\leq \frac{1}{20}n\log n$: 
$$
\fn{P}{Z'_t=1|\mc{S}_t\in \mc{G}'}
\ =\ \fn{P}{E_{(1,5)}^t|\mc{S}_t\in \mc{G}'}
\ \geq\  \frac{1}{220}\cdot \frac{1}{n}\cdot \left(1-\frac{\log^4 n}{n}\right)\cdot \frac{1}{22\log n}
\ \geq\ \frac{1}{5000 n\log n}
$$

Next denote the event to be $E'$. 
Since $\sum_{\tau=1}^{t} Z'_\tau$ is a counter, by Lemma~\ref{lem:Azuma}, we obtain the following bound:
\begin{align*}
    \fn{P}{Z'=0}
    & \leq  \ \sum_{S_{111\log^2(n)}} \fn{P}{E'}\cdot \fn{P}{\sum_{\tau=111\log^2 (n)}^{\frac{1}{20}n\log n} Z'_\tau=0\wedge \mc{S}_{\frac{1}{20}n\log n}\in {\cal{G}}'\middle|\mc{S}_{111\log^2 n}=S_{111\log^2 n}}\\
    & \leq \ \sum_{S_{111\log^2(n)}} \fn{P}{B_{\frac{1}{20}n\log n-111\log^2(n),\frac{1}{5000 n\log n}}=0}\\
    & = \ \fn{P}{\mc{S}_{\frac{1}{20}n\log n}\in {\cal{G}}'}
    \cdot \fn{P}{B_{\frac{1}{20}n\log n-111\log^2(n),\frac{1}{5000 n\log n}}=0}
\end{align*}

By Lemma~\ref{good-events-2}, $\fn{P}{\mc{S}_{\frac{1}{20}n\log n}\in {\cal{G}}'}=1-\smallo{1}$. So it is enough to show that:
$$\fn{P}{B_{\frac{1}{20}n\log n-111\log^2(n),\frac{1}{5000 n\log n}}=0} \ <\ 1-\Omega(1)$$
This follows as
\begin{align*}
  \fn{P}{B_{\frac{1}{20}n\log n-111\log^2(n),\frac{1}{5000 n\log n}}=0}
  &=\ \left(1-\frac{1}{5000 n\log n}\right)^{\frac{1}{20}n\log n-111\log^2(n)}\\
  &\leq\ \left(1-\frac{1}{5000 n\log n}\right)^{\frac{1}{21}n\log n}\\
  &\leq\ e^{-\frac{1}{105000}}\\
  &= \ 1-\Omega(1)\qedhere
\end{align*}
\end{proof}

We may now complete the proof of our second main result.
\begin{proof}[Proof of Theorem~\ref{main-thm-2}]
The probability that in a random instance there is a rotation of size $2$ is $\Omega(1)$ by Theorem~\ref{thm:small-rotation}. 
Given that there is a rotation of size $2$, the following is the probability that a random set of $\lambda n$ girls misses this rotation of size 2:
\begin{align*}
    \frac{\binom{n-2}{\lambda n}}{\binom{n}{\lambda n}}
    &=\ \frac{(n-2)!}{(n-2-\lambda n)!\cdot (\lambda n)!}\cdot\frac{(n-\lambda n)!\cdot (\lambda n)!}{n!}\\
    &=\ \frac{n-\lambda n}{n}\cdot \frac{n-1-\lambda n}{n-1}\\
    &\geq\  \left(1- \frac{\lambda n}{n-1}\right)^2.
\end{align*}

So, if $\lambda<1-\varepsilon$ for any positive constant $\varepsilon$, 
then the probability of missing a rotation is $\Omega(1)\cdot \left(1- \frac{\lambda n}{n-1}\right)^2=\Omega(1)$. 
\end{proof}

%% file: conclusion-Merged.tex
\section{Conclusion}
We have evaluated the expected cardinality of the minimum winning coalition.
We believe this result is of theoretical interest and that the techniques applied may have broader applications for
stable matching problems.
In terms of practical value it is worth discussing the assumptions inherent in the model.
The assumption of uniform and independent random preferences, while ubiquitous in the theoretical literature,
is somewhat unrealistic in real-world stable matching instances.
Furthermore, as presented, the model assumes full information, which is clearly not realistic in practice.
However, to implement the behavioural strategy presented in this paper, the assumption of full information 
is {\bf not} required. It needs only that a girl has a good approximation of the rank of her best stable partner.
But, by the results of Pittel~\cite{Pit89}, she does know this with high probability. Consequently, 
a near-optimal implementation of her behavioural strategy requires knowledge only of her own preference list!
This allows for a risk-free method to output a matching close in the lattice to the girl-optimal stable matching.
Similarly, as discussed, although our presentation has been in terms of a coalition of girls, each girl is able to
implement a near-optimal behavioural strategy independent of who the other girls in the coalition may be or what their 
preferences are.

%
%
%

%% file: main-Merged.bbl
\begin{thebibliography}{10}

\bibitem{AKL17}
I.~Ashlagi, Y.~Kanoria, and J.~Leshno.
\newblock Unbalanced random matching markets: the stark effect of competition.
\newblock {\em Journal of Political Economy}, 125(1):69--98, 2017.

\bibitem{Birk37}
G.~Birkhoff.
\newblock Rings of sets.
\newblock {\em Duke Mathematical Journal}, 3(3):443--454, 1937.

\bibitem{DF81}
L.~Dubins and D.~Freedman.
\newblock Machiavelli and the {G}ale-{S}hapley algorithm.
\newblock {\em American Mathematical Monthly}, 88(7):485--494, 1981.

\bibitem{FO90}
Philippe Flajolet and Andrew~M. Odlyzko.
\newblock Random mapping statistics.
\newblock In Jean-Jacques Quisquater and Joos Vandewalle, editors, {\em
  Advances in Cryptology --- EUROCRYPT '89}, pages 329--354, Berlin,
  Heidelberg, 1990. Springer Berlin Heidelberg.

\bibitem{GS62}
D.~Gale and L.~Shapley.
\newblock College admissions and the stability of marriage.
\newblock {\em American Mathematical Monthly}, 69(1):9--15, 1962.

\bibitem{Sot85}
D.~Gale and M.~Sotomayor.
\newblock Ms. {M}achiavelli and the stable matching problem.
\newblock {\em The American Mathematical Monthly}, 92(4):261--268, 1985.

\bibitem{Gon14}
Y.~Gonczarowski.
\newblock Manipulation of stable matchings using minimal blacklists.
\newblock In {\em Proceedings of the Fifteenth ACM Conference on Economics and
  Computation}, EC '14, page 449. Association for Computing Machinery, 2014.

\bibitem{Gus87}
D.~Gusfield.
\newblock Three fast algorithms for four problems in stable marriage.
\newblock {\em SIAM Journal on Computing}, 16(1):111--128, 1987.

\bibitem{GI89}
D.~Gusfield and R.~Irving.
\newblock {\em The Stable Marriage Problem: Structure and Algorithms}.
\newblock MIT Press, 1989.

\bibitem{IL86}
R.~Irving and P.~Leather.
\newblock The complexity of counting stable marriages.
\newblock {\em SIAM Journal on Computing}, 15(3):655--667, 1986.

\bibitem{IRL87}
R.~Irving, P.~Leather, and D.~Gusfield.
\newblock An efficient algorithm for the "optimal" stable marriage.
\newblock {\em J. ACM}, 34(3):532--543, 1987.

\bibitem{Knu82}
D.~Knuth.
\newblock Mariages stables et leurs relations avec d'autres probl\`emes
  combinatoires.
\newblock {\em Les Presses de l'Universit\'e de Montr\'eal}, 1982.

\bibitem{KMP90}
D.~Knuth, R.~Motwani, and B.~Pittel.
\newblock Stable husbands.
\newblock {\em Random Structures and Algorithms}, 1(1):1--14, 1990.

\bibitem{Kup18}
R.~Kupfer.
\newblock The instability of stable matchings: The influence of one strategic
  agent on the matching market.
\newblock {\em Proceedings of the 16th Conference on Web and Internet Economics
  (WINE)}, 2020.

\bibitem{MW71}
D.~McVitie and L.~Wilson.
\newblock The stable marriage problem.
\newblock {\em Communications of the ACM}, 14:486--490, 1971.

\bibitem{Pit89}
B.~Pittel.
\newblock The average number of stable matchings.
\newblock {\em SIAM Journal on Discrete Mathematics}, 2(4):530--549, 1989.

\bibitem{Pit92}
B.~Pittel.
\newblock On likely solutions of a stable matching problem.
\newblock {\em The Annals of Applied Probability}, 2(2):358--501, 1992.

\bibitem{Roth82}
A.~Roth.
\newblock The economics of matching: stability and incentives.
\newblock {\em Mathematics of Operations Research}, 7(4):617--628, 1982.

\bibitem{Stan97}
R.~Stanley.
\newblock {\em Enumerative Combinatorics, Volume I}.
\newblock Cambridge University Press, 1997.

\bibitem{Wil72}
L.~Wilson.
\newblock An analysis of the stable marriage assignment algorithm.
\newblock {\em BIT Numerical Mathematics}, 12:569--575, 1972.

\end{thebibliography}
